\documentclass[reqno]{amsart}
\usepackage{amsmath,amssymb,stmaryrd}
\usepackage{amsrefs}
\usepackage{amsfonts}
\usepackage{amsthm}
\usepackage{xifthen}
\usepackage{enumitem}
\usepackage{url}

\usepackage[colorlinks=true, pdfborder={0 0 0}]{hyperref}
\usepackage{cleveref}
\usepackage{upref}

\newtheorem{theorem}{Theorem}
\newtheorem{definition}{Definition}[section]
\newtheorem{proposition}{Proposition}[section]
\newtheorem{coro}{Corollary}[section]
\newtheorem{lemma}{Lemma}[section]

\def \rr {\mathbb{R}}
\def \rn {\mathbb{R}^n}
\def \nn {\mathbb{N}}

\def \rn {\mathbb{R}^n}
\def \tn {\mathbb{T}^n}

\def \eps {\epsilon}
\def \crit {{2^\star}}

\def \ua {u_\alpha}
\def \ue {u_\eps}
\def \hue {\hat{u}_\eps}

\newcommand\bpr[1]{\left(#1\right)}
\newcommand\bsq[1]{\left[#1\right]}
\newcommand\abs[1]{\left\lvert #1 \right\rvert}

\def \hua {\hat{u}_\alpha}
\def \ga {g_\alpha}
\def \va {v_\alpha}

\def \wa {w_\alpha}

\def \tga {\tilde{g}_\alpha}
\def \xa {x_\alpha}
\def \ma {\mu_\alpha}
\def \phia {\varphi_\alpha}

 \newcommand{\sub}{\subset}
\newcommand{\emb}{\hookrightarrow}
\newcommand{\Bal}[2]{B_{#2}(#1)}
\newcommand\Sob[1][k]{H^2_{#1}}
\newcommand{\hSob}[1][k]{D^2_{#1}}
\newcommand{\Cct}{C_c^\infty}
\newcommand{\norm}[1]{\left\| #1 \right\|}
\newcommand{\vprod}[3][g]{\langle #2,#3 \rangle_{#1}}
\newcommand\Snorm[3][(M)]{\norm{#3}_{\Sob[#2]#1}}
\newcommand\Hnorm[3][(\rr^n)]{\norm{#3}_{\hSob[#2]#1}}
\newcommand{\Lnorm}[3][]{\norm{#3}_{L^{#2}#1}}
\newcommand{\inorm}[1]{\norm{#1}_{*}}
\newcommand{\iinorm}[1]{\norm{#1}_{**}}
\newcommand\bigO{\mathcal{O}}
\newcommand{\intM}[2][M]{\int_{#1} #2\, dv_g}

\newcommand{\dg}[2][g]{d_{#1}{(#2)}}
\newcommand{\D}{\Delta}
\newcommand{\Dg}[1][g]{\D_{#1}}
\renewcommand{\a}{\alpha}
\newcommand{\pk}{a_{n,k}}
\newcommand{\pa}{\nu}
\newcommand{\zm}[1][]{z_{#1}, \mu_{#1}}
\newcommand{\zma}[1][]{\a_{#1}, \pa_{#1}}
\newcommand{\param}[1][]{\mathcal{P}(\rho_{#1},\a_{#1})}
\newcommand{\parama}[1][\tau_\a]{\mathcal{P}(#1,\a)}
\newcommand{\eBub}{U}
\newcommand{\Bub}{B}
\newcommand{\Ker}[1][\zma]{\mathcal{K}_{\ifthenelse{\isempty{#1}}{0}{#1}}}
\newcommand{\Keri}[1][i]{\mathcal{K}_{\zma[#1]}}
\newcommand{\Kera}[1][\pa_\a]{\mathcal{K}_{\a,#1}}
\newcommand{\proKe}[1][\zma]{\Pi_{\Ker[#1]^\perp}}
\newcommand{\proKa}[1][\pa_\a]{\Pi_{\Kera[#1]^\perp}}
\newcommand{\proKi}[1][i]{\Pi_{\Keri[#1]^\perp}}
\newcommand{\TBub}[1][]{V_{\zma[#1]}}
\newcommand{\TBuba}[1][\pa_{\a}]{V_{\a,#1}}
\newcommand{\Ze}[1]{Z^{#1}}
\newcommand{\Zed}[2][]{Z^{#2}_{\zma[#1]}}
\newcommand{\Zeda}[2][\pa_{\a}]{Z^{#2}_{\a,#1}}

\newcommand{\Lia}[1][]{L_{\zma[#1]}}
\newcommand{\Tet}[1][]{\Theta_{\zma[#1]}}
\newcommand{\Teta}[1][\pa_{\a}]{\Theta_{\a,#1}}
\newcommand{\Psp}{\Psi^{(p)}}

\newcommand{\Tg}{\Tilde{g}}
\newcommand{\Tp}{\Tilde{\phi}}

\newcommand{\bmu}{\Bar{\mu}}
\newcommand{\bpa}{\Bar{\pa}}

 \def \eps {\epsilon}
\title[Optimal Sobolev inequalities of high order with $L^2-$remainder]{Optimal Sobolev inequalities of high order with $L^2-$remainder}

\author{Lorenzo Carletti}
\address{Lorenzo Carletti, Universit\'e Libre de Bruxelles, Service d'analyse, CP 213, Boulevard du Triomphe, B-1050 Bruxelles, Belgique.}
\email{lorenzo.carletti@ulb.be}

\author{Fr\'ed\'eric Robert}
\address{Fr\'ed\'eric Robert, Universit\'e de Lorraine, CNRS, IECL, F-54000 Nancy, France}
\email{frederic.robert@univ-lorraine.fr}

\date{June 20th, 2025}


 \addtocontents{toc}{\protect\setcounter{tocdepth}{1}}

\begin{document}
\begin{abstract} We investigate the validity of the optimal higher-order Sobolev inequality $H_k^2(M^n)\hookrightarrow L^{\frac{2n}{n-2k}}(M^n)$ on a closed Riemannian manifold  when the remainder term is the $L^2-$norm. Unlike the case $k=1$, the optimal inequality does not hold in general for $k>1$. We prove conditions for the validity and non-validity that depend on the geometry of the manifold. Our conditions are sharp when $k=2$ and in small dimensions.
\end{abstract}

\maketitle
%\tableofcontents
\section{Introduction and main results}
Let $k,n\in\nn$ be such that $2\leq 2k<n$ and $(M,g)$ be a compact Riemannian manifold of dimension $n$ without boundary. We define the Sobolev space $H_m^2(M)$ as the completion of $C^\infty(M)$ for the norm $u\mapsto \Vert u\Vert_{H_m^2}:=\sum_{i=0}^m\Vert \nabla^i u\Vert_2$. Sobolev's embedding theorem (see Aubin \cite{aubin} or Hebey \cite{hebey.cims} for an exposition in book form) yields that there exist $A,B>0$ such that 
\begin{equation}\label{ineq:AB}
\left(\int_M|u|^{\crit}\, dv_g\right)^{\frac{2}{\crit}}\leq A\int_M(\Delta_g^{k/2}u)^2\, dv_g +B\Vert u\Vert_{H_{k-1}^2}^2\hbox{ for all }u\in H_k^2(M),
\end{equation}
where $\crit:=\frac{2n}{n-2k}$, $dv_g$ is the Riemannian element of volume, $\Delta_g:=-\hbox{div}_g(\nabla)$ is the Riemannian Laplacian and $\Delta_g^{\frac{p}{2}}:=\nabla\Delta_g^l$ if $p=2l+1$ is odd. Mazumdar \cite{mazumdar:jde} proved that one can take $A$ as close as wanted to the Euclidean constant $K(n,k)>0$ defined as 
\begin{equation}\label{def:K}
\frac{1}{K(n,k)}:=\inf_{u\in D_k^2(\rn)-\{0\}}\frac{\int_{\rn}(\Delta^{k/2}_\xi u)^2\, dx}{\left(\int_{\rn}|u|^{\crit}\, dx\right)^{\frac{2}{\crit}}}
\end{equation}
where $\xi$ denote the Euclidean metric on $\rn$, $D_k^2(\rn)$ is the completion of $C^\infty_c(\rn)$ for the norm $u\mapsto \Vert u\Vert_{D_k^2}:=\Vert\Delta_\xi^{k/2}u\Vert_2$. The value of $K(n,k)$ and the extremals for \eqref{def:K} are explicit as shown by Swanson \cite{swanson}.
%In particular, for all $\eps>0$, there exists $B_\eps>0$ such that
%\begin{equation}\label{eq:epsineq}
%\left(\int_M|u|^{\crit}\, dv_g\right)^{\frac{2}{\crit}}\leq (K(n,k)+\eps)\int_M(\Delta_g^{k/2}u)^2\, dv_g +B_\eps  \Vert u\Vert_{H_{k-1}^2}^2\hbox{ for all }u\in H_k^2(M).
%\end{equation}
Recently, using sharp blow-up analysis Carletti \cite{Car24}, and independently Zeitler \cite{zeitler} have proved that one can take exactly $A=K(n,k)$ in \eqref{ineq:AB}, that is there exists $B>0$ such that
\begin{equation}\label{ineq:opt:k}
\left(\int_M|u|^{\crit}\, dv_g\right)^{\frac{2}{\crit}}\leq  K(n,k) \int_M(\Delta_g^{k/2}u)^2\, dv_g +B_0\Vert u\Vert_{H_{k-1}^2}^2\hbox{ for all }u\in H_k^2(M).
\end{equation}
See Hebey-Vaugon \cite{hebeyvaugon} for the case $k=1$, Hebey \cite{hebey:2003} for $k=2$, and Druet \cite{druet:ma} for $L^p-$inequalities.
In particular, the validity of this optimal inequality is independent of the geometry of $(M,g)$. 

Testing \eqref{ineq:opt:k} against the constant functions, we see that it is necessary to keep a lower-order Lebesgue norm. In this paper, we investigate whether the optimal Sobolev inequality holds when we replace the $H_{k-1}^2-$norm by the $L^2-$norm, which is the lowest natural Hilbert norm possible. This follows the spirit of the AB-programme of Druet-Hebey \cite{druet:hebey:AB}.
% , we investigate the validity of inequality \eqref{ineq:opt:k} . 
By interpolating, it follows from \eqref{ineq:opt:k} and the compact embedding of $H^2_{k}(M)\hookrightarrow H^2_{k-1}(M)$ that for all $\eps>0$, there exists $B_\eps>0$ such that 
\begin{equation}\label{eq:epsineq}
\left(\int_M|u|^{\crit}\, dv_g\right)^{\frac{2}{\crit}}\leq (K(n,k)+\eps)\int_M(\Delta_g^{k/2}u)^2\, dv_g +B_\eps \int_M u^2\, dv_g
\end{equation}
for all $u\in H_k^2(M)$. We say that \eqref{ineq:opt} holds if there exists $B>0$ such that 
\begin{equation}\label{ineq:opt}
\left(\int_M|u|^{\crit}\, dv_g\right)^{\frac{2}{\crit}}\leq  K(n,k) \int_M(\Delta_g^{k/2}u)^2\, dv_g +B\int_Mu^2\, dv_g\tag{$I_{opt,k}$}
\end{equation}
for all $u\in H_k^2(M)$. When $k=1$, $(I_{opt,1})$ is valid since it is exactly \eqref{ineq:opt:k}. So we focus on $k\geq 2$. Surprisingly, in this situation, the geometry plays a role in the validity of \eqref{ineq:opt}, in striking contrast with \eqref{ineq:opt:k}: 
\begin{theorem}\label{th:main} Let $(M,g)$ be a compact Riemannian manifold of dimension $n$ and let $k\in\nn$ such that $n> 2k\geq 4$. 
\begin{itemize}
\item Assume that $n\geq 2k+2$ and that the scalar curvature $R_g$ is positive somewhere. Then \eqref{ineq:opt} does not hold.
\item Assume that $n=2k+1$ or that $\{n\geq 2k+2\hbox{ and } R_g<0\hbox{ everywhere}\}$. Then \eqref{ineq:opt} holds.
\end{itemize}\end{theorem}
We answer completely the question of the validity of \eqref{ineq:opt} when $k=2$:
\begin{theorem}\label{th:k=2} Let $(M,g)$ be a compact Riemannian manifold of dimension $n\geq 5$. Then 
%
%\begin{equation*}
%\left(\int_M|u|^{\frac{2n}{n-4}}\, dv_g\right)^{\frac{n-4}{n}}\leq  K(n,2) \int_M(\Delta_g u)^2\, dv_g +B\int_Mu^2\, dv_g\hbox{ for all }u\in H_2^2(M).\eqno{(I_{opt,2})}.
%\end{equation*}
%then
$$\{(I_{opt,2})\hbox{ holds }\}\Leftrightarrow \{n=5\hbox{ or }R_g\leq 0\hbox{ everywhere in }M\}.$$
\end{theorem}

In view of the case $k=2$, it is a natural guess that the validity of \eqref{ineq:opt} could depend   on the sole scalar curvature. However, this is not true as shown in Theorem \ref{th:ok:noIopt:bis} below. Indeed, when $n\geq 2k+4$ and $k>2$, there are Ricci-flat manifolds (and therefore scalar-flat) for which \eqref{ineq:opt} is valid, and some for which it is not. We refer to Section \ref{sec:others} for other results, in particular sharp conditions when $n=2k+2,2k+3$.  

\medskip\noindent 
One fascinating feature in the study of \eqref{ineq:opt} is that the geometry of the manifold contributes to the problem at scales independent of the order $k$ (see  \eqref{exp:2} and \eqref{eq:compMV}). On the contrary, the $L^2$ remainder term in \eqref{ineq:opt} impacts the problem at scales depending on the order. Thus, the interaction between these terms is completely different depending on whether $k$ is small or not. We refer to Section \ref{sec:test} for insights on this issue via test-function computations.

\medskip\noindent An interpretation of these results is that the Sobolev inequalities always hold on a manifold with the same optimal constant as for $\rn$ up to the addition of a natural remainder (here, it is $\Vert\cdot\Vert_{H_{k-1}^2}$). However, if we take a lower-order term as the remainder (here the $L^2-$norm), there are geometric conditions that arise to satisfy the optimal inequality. A similar phenomenon has been observed by Druet-Hebey-Vaugon \cite{druet:hebey:vaugon} and Druet-Hebey \cite{druet:hebey:SP}  who considered $k=1$ and took the $L^1-$norm instead of the natural $L^2-$norm, and by Druet-Hebey-Vaugon \cite{dhv:nash} for the Nash inequality.

\medskip\noindent 
The method developed in this article can be adapted to consider another remainder term in (3). In light of the analysis we develop here, we expect the same type of result to hold when taking any intermediate lower-order norm like the $H^2_l$ norm, for $1\leq l\leq k-2$. This issue also makes sense and can be approached by our method for other functional inequalities like

% This type of result is expected when taking an intermediate remainder term like the $H_{k-2}^2-$norm or the $H_1^2-$norm when $k\geq 3$. This issue also makes sense for other functional inequalities like 
\begin{equation*}
\left(\int_M|u|^{\crit}\, dv_g\right)^{\frac{2}{\crit}}\leq  K(n,k) \int_MuP_{g}^ku\, dv_g +B\int_M u^2\, dv_g
\end{equation*}
for all $u\in H_k^2(M)$, where $P_g^k$ is the conformal GJMS operator introduced by Graham-Jenne-Mason-Sparling \cite{gjms}  of order $2k$ (see \eqref{def:gjms} below). This has been investigated in Djadli-Hebey-Ledoux \cite{dhl} for $k=2$.

\medskip\noindent 
The method we develop works for elliptic equations independently of the order. Previously, the techniques for $k=1$ and $k=2$ relied heavily on the order of the equation arising from the problem. 
Our new approach allows to bypass the previous obstacles by obtaining an improved pointwise description of sequences of solutions to polyharmonic equations. 
Recently, there has been a growing interest on quantitative stability questions, we refer to Deng-Sun-Wei \cite{dengsunwei}, Figalli-Glaudo \cite{figa:glaud}, Nobili-Parise \cite{nob:parisi} for second-order operators. Another area of interest regards compactness of $Q$-curvature conformal metrics: we refer to Khuri-Marques-Schoen \cite{kms} for $k=1$ and Li-Xiong \cite{lx} for $k=2$.
Our method opens new perspectives on such elliptic problems of arbitrary higher order, without positivity assumptions on the operator. 

\medskip\noindent
The proofs of Theorems \ref{th:main} and \ref{th:k=2} rely on the following pointwise control \eqref{eq:estimbndua}, a variant of which was first obtained in Carletti \cite{Car24} in order to prove \eqref{ineq:opt:k}.
\begin{theorem}\label{th:bndua} Let $(M,g)$ be a compact Riemannian manifold of dimension $n$ and let $k\in\nn$ such that $n> 2k>2$. Let $(u_\a)_\a \in C^{2k}(M)$ be a sequence of  solutions $u_\a \not\equiv 0$ to 
    \[  \Dg^k u_\a + \a^{2k} u_\a = |u_\a|^{\crit-2} u_\a \qquad \text{in } M, \quad \forall~\a >0,
        \]
    satisfying 
    \[  \int_M{|u_\a|^\crit} \leq K(n,k)^{-\frac{n}{2k}}+o(1).
        \]
    Then there exists $(x_\a)_\a \in M$, $(\mu_\a)_\a \in \rr_{>0}$ such that $\a\mu_\a  \to 0$ as $\a \to \infty$ and  for all $p \geq 1$, there exists $C_p>0$ such that for all $l=0,...,2k-1$
%\begin{equation}\label{cv:ua:resc}
% \mu_\a^{\frac{n-2k}{2}} u_\a(\exp_{z_\a}(\mu_\a \cdot)) \to \eBub:= \left(\frac{1}{1+a_{n,k}|\cdot|^2}\right)^{\frac{n-2k}{2}} \hbox{ in } C^{2k}_{loc}(\rr^n).
% \end{equation}   
% where $a_{n,k}:=\left(\Pi_{j=-k}^{k-1}(n+2j)\right)^{-\frac{1}{k}}$. 
     \begin{equation}\label{eq:estimbndua}  
        | \nabla^l u_\a(x)| \leq \frac{C_p}{ \big(1+ \a  \dg{z_\a,x} \big)^{p}}\cdot \frac{\mu_\a^{\frac{n-2k}{2}}}{\big(\mu_\a + \dg{x_\a,x}\big)^{n-2k+l}} \qquad \text{for all } x \in M.
        \end{equation}
\end{theorem}
Pointwise estimates for second-order equations have been obtained before, see Hebey-Vaugon \cite{hebeyvaugon}, Druet-Hebey \cite{DH} for positive solutions, Premoselli \cite{Pre24} and Premoselli-Robert \cite{PreRob} for sign-changing solutions, see also Robert \cite{robert:gjms} for higher-orders. The method of simple blow-up points yields similar controls for positive solutions, see Khuri-Marques-Schoen \cite{kms} and Li-Xiong \cite{lx}. 
One striking feature of our pointwise control is the fast decay at infinity. This follows from the decay of the Green's function studied in section \ref{sec:green:M}, which is specific to the case of diverging coefficients. A similar decay was obtained in Carletti \cite{Car24}.

\medskip\noindent Theorems \ref{th:main} and \ref{th:k=2} are direct consequences of Theorems \ref{th:ok:Iopt:bis} and \ref{th:ok:noIopt:bis} below. Theorem \ref{th:bndua} is a direct consequence of Theorem \ref{th:bndua:der}. In this paper, we prove Theorems \ref{th:ok:Iopt:bis}, \ref{th:ok:noIopt:bis}  and \ref{th:bndua:der}.  In Section \ref{sec:test}, we perform test-functions estimates in order to prove the non-validity in Theorem \ref{th:ok:noIopt:bis}. In Section \ref{sec:part4}, we prove the pointwise control Theorem \ref{th:bndua:der} by using a constructive approach inspired by Premoselli \cites{Pre22, Pre24} and Carletti \cite{Car24}. This constructive approach is the object of Theorem \ref{prop:unicC0} that is proved in Section \ref{sec:part3}. In Section \ref{sec:blowup}, we perform a blow-up analysis to prove the validity of \eqref{ineq:opt} in Theorem \ref{th:ok:Iopt:bis}. The last section is devoted to the properties of the Green's function for the polyharmonic operator $\Delta_g^k+\alpha^{2k}$ on $\rn$ and on a compact manifold.

\section{Extensions and miscellaneous}\label{sec:others}
\subsection{Some sharp results for the validity of \eqref{ineq:opt}}
Theorems \ref{th:main} and \ref{th:k=2} above are consequences of the following more general results:

\begin{theorem}\label{th:ok:Iopt:bis}  Let $(M,g)$ be a compact Riemannian manifold of dimension $n$ and let $k\in\nn$ such that $n>2k$. Assume that $k>1$
\begin{itemize}
\item If $n=2k+1$, then \eqref{ineq:opt} holds regardless of the geometry
\item Assume that $n\in \{2k+2,2k+3\}$ or $\{k= 2\hbox{ and }n\geq 6\}$: then \{\eqref{ineq:opt} holds \} $\Leftrightarrow$ $\{R_g\leq 0\hbox{ everywhere}\}$
 \end{itemize} \end{theorem}

\begin{theorem}\label{th:ok:noIopt:bis}  Let $(M,g)$ be a compact Riemannian manifold of dimension $n$ and let $k\in\nn$ such that $n>2k$. Assume that $k>1$ and $n\geq 2k+2$
\begin{itemize}
\item \eqref{ineq:opt} holds if $R_g<0$ everywhere;
\item \eqref{ineq:opt} does not holds if $R_g(x)>0$ somewhere;
\item \eqref{ineq:opt} holds if $Rm_g\equiv 0$ everywhere (for example the flat torus) 
\item If in addition, we assume that $k>2$ and $n\geq 2k+4$, then \eqref{ineq:opt} does not holds if $(M,g)$ is Ricci-flat but not flat.
\end{itemize} \end{theorem}
Remarks:
\begin{itemize}
\item There are examples of Ricci-flat manifolds that are non-flat in arbitrary high dimension (see Yau \cite{yau})
\item Unlike the case $k=2$ or $n\leq 2k+3$, the last example shows that information on the sole scalar curvature is not enough to decide the validity of \eqref{ineq:opt} in the other situations. In order to go further one would need to assume a sufficient vanishing order of the Weyl tensor as in Druet-Hebey-Vaugon \cite{druet:hebey:vaugon}, Druet-Hebey \cite{druet:hebey:SP}, and Hebey-Vaugon \cite{hebey:vaugon:equiv}.
%\item dire un mot dans l'intro sur le terme en weyl qui intervient a l'orde 4 [xxx] et donc imporytance de k=2
\end{itemize}

\subsection{Refined approximation by the bubble}
We define
\begin{equation}\label{def:U}
 U(X):= \left(\frac{1}{1+a_{n,k}|X|^2}\right)^{\frac{n-2k}{2}} \hbox{ for all  } X\in\rn,
 \end{equation}   
where $a_{n,k}:=\left(\Pi_{j=-k}^{k-1}(n+2j)\right)^{-\frac{1}{k}}$. It follows from Swanson \cite{swanson} that $U\in D_k^2(\rn)$ is an extremal for \eqref{def:K} and satisfies $\Delta_\xi^kU=U^{\crit-1}$. Indeed, all extremals for \eqref{def:K}  arise from $U$ \cite{swanson}, and all positive solutions to $\Delta_\xi^kV=V^{\crit-1}$ are equal to $U$ up to conformal transformation (Wei-Xu \cite{WeiXu99}).   We fix $N\in\nn$. It follows from Lee-Parker \cite{lp} that for all $p\in B_{\delta}(x_0)$, there exists $\varphi_p\in C^\infty(M)$, $\varphi_p>0$, such that, setting $g_p:=\varphi_p^{\frac{4}{n-2k}}g$, we have that 
\begin{equation}\label{def:phi}
\varphi_p(p)=1\, ,\, \nabla\varphi_p(p)=0\, ,\, \hbox{Ric}_{g_p}(p)=0\hbox{ and }dv_{g_p}=\sqrt{|g_p|}\, dx=(1+O(|x|^N))\, dx
\end{equation}
Moreover, the map $p\mapsto \varphi_p$ is continuous. We fix once for all $\delta < i_{g_p}/2$ for all $p\in M$, where $i_h>0$ is the injectivity radius of the manifold $(M,h)$. 
\begin{definition}
    For $z\in M$ and $\mu>0$, we define 
    \[  \Bub_{\zm}(x) = \bpr{\frac{\mu}{\mu^2 + \pk d_{g_z}(z,x)^2}}^{\frac{n-2k}{2}} \qquad \text{for }x \in M.
        \]
\end{definition}
For $x \in \Bal{z}{\delta}$, we then have $\Bub_{\zm}(x) = \mu^{-\frac{n-2k}{2}} \eBub\Big(\tfrac{1}{\mu}(\exp_z^{g_z})^{-1}(x)\Big)$ where the exponential map is taken with respect to the metric $g_z$. Given $\a >0$ and $\rho\leq 1$, we define the parameter set 
    \[  \param := \{ (\zm) \in M\times (0,+\infty) \,;\, \a\mu  \leq \rho\}.
        \]

\begin{definition}
    Let $\a \geq \delta^{-1}$ and $\rho\leq 1$, $\pa = (\zm) \in \param$ and let $\chi \in \Cct(\rr_{\geq 0})$ be such that $\chi \equiv 1$ in $[0,1)$ and $\chi \equiv 0$ in $(2,+\infty)$ %and for any $a>0$, $\chi^\prime\chi^{a-1}$ is uniformly bounded in $\rr$ 
    we define 
    \begin{equation}\label{def:TBub}  
        \TBub(x) := \chi(\a d_{g_z}(z,x))\, \varphi_z(x)\, \Bub_{\zm}(x) \qquad \text{for all } x \in M.
        \end{equation}
\end{definition}
%[xxx]The assumption on boundedness will be used  to ensure that $\TBub^{\crit-2}\in C^1(M)$. but indeed we do not it need it anymore[xxx]

\smallskip\noindent Theorem \ref{th:bndua} is a consequence of the following more precise result:
\begin{theorem}\label{th:bndua:der} Let $(M,g)$ be a compact Riemannian manifold of dimension $n$ and let $k\in\nn$ such that $n> 2k >2$. Let $(u_\a)_\a \in C^{2k}(M)$ be a sequence of  solutions $u_\a \not\equiv 0$ to 
\begin{equation}\label{eq:ua:th22}\Dg^k u_\a + \a^{2k} u_\a = |u_\a|^{\crit-2} u_\a \qquad \text{in } M, \quad \forall~\a >0,
\end{equation}
    satisfying 
\begin{equation}\label{bnd:nrj}
\int_M{|u_\a|^\crit} \leq K(n,k)^{-\frac{n}{2k}}+o(1).
\end{equation}
Then there exists $(x_\a)_\a \in M$, $(\mu_\a)_\a \in \rr_{>0}$ such that $\a\mu_\a  \to 0$ as $\a \to \infty$, and such that
\begin{equation}\label{cv:ua:resc}
 \mu_\a^{\frac{n-2k}{2}} u_\a(\exp_{x_\a}(\mu_\a \cdot)) \to \eBub:= \left(\frac{1}{1+a_{n,k}|\cdot|^2}\right)^{\frac{n-2k}{2}} \hbox{ in } C^{2k}_{loc}(\rr^n).
 \end{equation}
 Moreover, for all $p \geq 1$, and for all $\tau \in \rr$ such that $0<\tau\leq\min\left\{2,\frac{n-2k}{2}\right\}$, there exists $C_{p,\tau}>0$ such that
    \begin{equation}\label{eq:estimbndua:der}  
        (\mu_\a+ \dg{x_\a,x})^{l}|\nabla_g^l \left(u_\a -V_{\alpha, (x_\alpha,\ma)}\right)(x)| \leq C_{p,\tau}\frac{(\alpha(\ma+d_g(x_\a,x)))^\tau}{(1+\alpha d_g(\xa, x))^p}B_{\xa,\ma}(x)
        \end{equation}
        for all $x\in M$, and $l=0,...,2k-1$.
\end{theorem}

\section{Theorems \ref{th:ok:Iopt:bis} and  \ref{th:ok:noIopt:bis}: proof of   non-validity via test-functions}\label{sec:test}
We use the notations of Robert \cite{robert:gjms}. We consider an elliptic operator
\begin{equation}
P=\Delta_g^k + \Delta_g^{\frac{k-1}{2}}(h\Delta_g^{\frac{k-1}{2}})-\Delta_g^{k-2}((T ,\nabla_g^2))+\sum_{i=0}^{k-2} (-1)^i\nabla_g^i(A^{(i)}  \nabla_g^i)\label{def:P}
\end{equation}
where $h\in C^{k-1}(M)$, $T$ is a smooth field of symmetric $(2,0)-$tensors and for all $i=0,...,k-2$, $A^{(i)}$ is a smooth symmetric $(0, 2i)-$tensor. When $k=1$, we take $T\equiv 0$ and the sum is empty. We fix $x_0\in M$ and we define
$$\hat{u}_\eps(x):=\eta(x)\left(\frac{\eps}{\eps^2+a_{n,k}d_{g_{x_0}}(x,x_0)^2}\right)^{\frac{n-2k}{2}}=\eta(x) \Bub_{x_0,\eps}(x) \hbox{ for all }x\in M$$
where $\eta\in C^\infty(M)$ is such that $\eta(x)= 1$ for $x\in B_\delta(x_0)$ and $\eta(x)= 0$ for $x\in M\setminus B_{2\delta}(x_0)$. We set  $\ue:=\varphi_{x_0}\hat{u}_\eps,$
where $\varphi_{x_0}$ is the conformal factor defined in \eqref{def:phi}.
The class of operators of type $P$ includes the well known GJMS operator $P_g^k$, 
\begin{equation}\label{def:gjms}
P_g^k=\Delta_g^k+\sum_{i=0}^{k-1} (-1)^i\nabla_g^i(A_{g}^{(i)} \nabla_g^i),
\end{equation}
that enjoys invariance under conformal changes of metrics. Namely, for $\omega\in C^\infty(M)$, $\omega>0$, taking $\tilde{g}:=\omega^{\frac{4}{n-2k}}g$ a metric conformal to $g$, we have that
\begin{equation}\label{invar:gjms}
P_{\tilde{g}}^k\varphi=\omega^{1-\crit}P_g^k(\omega\varphi)\hbox{ for all }\varphi\in C^\infty(M).
\end{equation}
\subsection{The case of general manifolds} Testing $Pu_\epsilon$ against $u_\epsilon$, where $P$ is as in \eqref{def:P}, the conformal property \eqref{invar:gjms} yields
\begin{eqnarray}\label{eq:upu}
\int_M \ue P\ue\, dv_g &=& \int_M \hat{u}_\eps\Delta^k_{g_{x_0}}\hue\, dv_{g_{x_0}}+\sum_{i=0}^{k-1}\int_M \hat{A}^{(i)} ( \nabla^i_{g_{x_0}}\hue, \nabla^i_{g_{x_0}}\hue)\, dv_{g_{x_0}}
\end{eqnarray}
where $\hat{A}^{(i)}$ is a smooth symmetric $(0, 2i)-$tensor for all $i=0,...,k-1$ and
where, see Subsection 15.2 in \cite{robert:gjms}
\begin{eqnarray*}%\label{eq:76}
\hat{A}^{(k-1)}( \nabla^{k-1}_{g_{x_0}}\hat{u}, \nabla^{k-1}_{g_{x_0}}\hat{u})&:=&\hat{h} (\Delta_{g_{x_0}}^{\frac{k-1}{2}} \hat{u})^2-(\hat{T},\nabla_{g_{x_0}}^2\hat{u})\Delta_{g_{x_0}}^{k-2}\hat{u}+lot
\end{eqnarray*}
where $\hat{h}:=\varphi_{x_0}^{-\frac{4}{n-2k}}h $ and $\hat{T}:=\varphi_{x_0}^{-\frac{8}{n-2k}} (T-T_g)+ T_{g_{x_0}}$ and, see (2.7) in Mazumdar-Vétois \cite{mv}, 
\begin{equation}\label{def:T:g}
T_g:=\frac{k(n-2)}{4(n-1)}R_g g-\frac{2k(k-1)(k+1)}{3(n-2)}\left(\hbox{Ric}_g-\frac{R_g}{2(n-1)}g\right).
\end{equation}
We compute separately the different terms of \eqref{eq:upu}. We have that for all $X\in B_{\eps^{-1}\delta}(0)$, $\eps^{\frac{n-2k}{2}}\hue(\hbox{exp}_{x_0}^{g_{x_0}}(\eps X))=U(X)$. Note also that $\Delta_\xi^kU=U^{\crit-1}$ in $\rn$ and $U\in D_k^2(\rn)$ is an extremal for \eqref{def:K}. For convenience, we define
$$\theta_\eps:=\left\{\begin{array}{cc}
\eps^2 &\hbox{ if }n>2k+2,\\
\eps^2\ln\frac{1}{\eps} &\hbox{ if }n=2k+2.
\end{array}\right.$$
We have that $\int_M \hat{u}_\eps\Delta^k_{g_{x_0}}\hue\, dv_{g_{x_0}}= \int_{B_\delta(x_0)} \hat{u}_\eps\Delta^k_{g_{x_0}}\hue\, dv_{g_{x_0}}+O(\eps^{n-2k})$
as $\eps\to 0$. Since $\hue$ is radially symmetrical in $B_\delta(x_0)$, in the exponential chart at $x_0$, we get
$$\Delta_{g_{x_0}}\hue=-\frac{1}{r^{n-1}\sqrt{|g_{x_0}|}}\partial_r\left(r^{n-1}\sqrt{|g_{x_0}|} \partial_r\hue\right)=\Delta_\xi\hue-\frac{\partial_r \sqrt{|g_{x_0}|}}{\sqrt{|g_{x_0}|}}\partial_r\hue.$$
Taking $N$ large enough in \eqref{def:phi} and iterating this formula, we get that 
\begin{equation}\label{diff:riem:eucl}
\Delta^k_{g_{x_0}}\hue=\Delta^k_{\xi}(\hue\circ \hbox{exp}^{g_{x_0}}_{x_0})+O(\eps^{\frac{n-2k}{2}}r^{n-2k}).
\end{equation}
Therefore, with a change of variable, we get that 
\begin{equation*}
 \int_M \hat{u}_\eps\Delta^k_{g_{x_0}}\hue\, dv_{g_{x_0}} 
 = \int_{B_{\eps^{-1}\delta}( 0)} U\Delta^k_{\xi}U \, dx+O(\eps^{n-2k})=\int_{\rn} U^{\crit} \, dx+O(\eps^{n-2k}).
% \int_{B_{\eps^{-1}\delta}( 0)} U^{\crit} \, dx+O(\eps^{n-2k})\\
% &=&\int_{\rn} U^{\crit} \, dx+O(\eps^{n-2k}).
\end{equation*}
Straightforward computations yield $\int_M \hat{A}^{(i)} ( \nabla^i_{g_{x_0}}\hue, \nabla^i_{g_{x_0}}\hue)\, dv_{g_{x_0}}=o(\theta_\eps)$ when $i<k-1$. We have that
\begin{multline*}
    \int_M \hat{A}^{(k-1)} ( \nabla^{k-1}_{g_{x_0}}\hue, \nabla^{k-1}_{g_{x_0}}\hue)\, dv_{g_{x_0}}\\= \int_M \hat{h} (\Delta_{g_{x_0}}^{\frac{k-1}{2}} \hue)^2\, dv_{g_{x_0}}-\int_M (\hat{T},\nabla_{g_{x_0}}^2\hue)\Delta_{g_{x_0}}^{k-2}\hue\, dv_{g_{x_0}}+o(\theta_\eps)\end{multline*}
The computations of Subsection 15.2 in \cite{robert:gjms} yield
\begin{equation*}
\int_M \hat{A}^{(k-1)} ( \nabla^{k-1}_{g_{x_0}}\hue, \nabla^{k-1}_{g_{x_0}}\hue)\, dv_{g_{x_0}} = C_{n,k}\theta_\eps\left(h_0(x_0)+\frac{\hbox{Tr}_{x_0}(T-T_g)}{n}\right) \, dx +o(\theta_\eps).
\end{equation*}
%\begin{eqnarray*}
%&&\int_M \hat{A}^{(k-1)} ( \nabla^{k-1}_{g_{x_0}}\hue, \nabla^{k-1}_{g_{x_0}}\hue)\, dv_{g_{x_0}}\\
%&&=\eps^2\left(h_0(x_0)+\frac{\hbox{Tr}_{x_0}(T-T_g)}{n}\right)\left\{\begin{array}{cc}
%\int_{\rn} (\Delta_{\xi}^{\frac{k-1}{2}}U )^2 \, dx +o(\eps^2)&\hbox{ if }n>2k+2\\
% C_n\ln\frac{1}{\eps } +o\left(\eps^2\ln\frac{1}{\eps}\right)&\hbox{ if }n=2k+2
%\end{array}\right.
%\end{eqnarray*}
for
\[ C_{n,k} =\begin{cases}
    \int_{\rn} (\Delta_{\xi}^{\frac{k-1}{2}}U )^2 & \hbox{if } n>2k+2\\
    \omega_{n-1}\lim_{r\to +\infty}r^{n-1} (\Delta_{\xi}^{\frac{k-1}{2}}U (r))^2& \hbox{if } n=2k+2 
\end{cases}.
\]
%\begin{equation}\label{def:Cn}
%C_n:=\omega_{n-1}\lim_{r\to +\infty}r^{n-1} (\Delta_{\xi}^{\frac{k-1}{2}}U (r))^2>0\hbox{ when }n=2k+2.
%\end{equation}
We have that (see \cite{robert:gjms})
%$$\frac{\hbox{Tr}_{x_0}( T_g)}{n}=\frac{k(3n(n-2)-4(k^2-1))}{12n(n-1)}R_g(x_0),$$
%so that
$$h_0(x_0)+\frac{\hbox{Tr}_{x_0}(T-T_g)}{n}=h_0(x_0)-c_{n,k}R_g(x_0)+\frac{\hbox{Tr}_{x_0}(T)}{n}$$
where
\begin{equation}\label{def:petit:c}
c_{n,k}:=\frac{k(3n(n-2)-4(k^2-1))}{12n(n-1)}.
\end{equation}
These computations have been performed for $k>1$. The result is still valid for $k=1$ and $T\equiv 0$. Finally, putting these identities together yield
\begin{equation*}
\int_M \ue P\ue\, dv_g = \int_{\rn} U^{\crit} \, dx  \\
+\left(h_0(x_0)-c_{n,k}R_g(x_0)+\frac{\hbox{Tr}_{x_0}(T )}{n}\right)C_{n,k}\theta_\eps+o(\theta_\eps).
\end{equation*}
Independently, taking also $N>>1$, we get that
\begin{eqnarray*}
&&\int_M |\ue|^{\crit} \, dv_g =  \int_{B_\delta(x_0)} |\hue|^{\crit} \, dv_{g_{x_0}}+O(\eps^{n})=\int_{\rn} U^{\crit} \, dx+O(\eps^{n})
%&&=\int_{B_\delta(0)} |\hue\circ\hbox{exp}^{g_{x_0}}_{x_0}|^{\crit} (1+O(|x|^N))\, dx+O(\eps^{n})\\
%&&= \int_{B_{\eps^{-1}\delta(0)}} U^{\crit} \, dx+O(\eps^{n})=\int_{\rn} U^{\crit} \, dx+O(\eps^{n})
\end{eqnarray*}
Therefore, using that $U$ is an extremal for \eqref{def:K}, we get that
\begin{equation}\label{final:upu}\begin{aligned}
&\frac{\int_M \ue P\ue\, dv_g}{\left(\int_M|\ue|^{\crit} \, dv_g\right)^{\frac{2}{\crit}}} =\frac{1}{K(n,k)}+ \left(h_0(x_0)-c_{n,k}R_g(x_0)+\frac{\hbox{Tr}_{x_0}(T)}{n}\right)\frac{C_{n,k}}{\left(\int_{\rn}U^{\crit} \, dX\right)^{\frac{2}{\crit}}} \theta_\eps+o(\theta_\eps).
\end{aligned}\end{equation}
We can now conclude the proof of the nonexistence part of Theorem \ref{th:main}. Assume that there exists $B_0>0$ such that \eqref{ineq:opt} holds. Assume that there exists $x_0\in M$ such that $R_g(x_0)>0$. We define $B'_0:=K(n,k)^{-1}B_0$, so that
\begin{equation}\label{mino:K}
\frac{1}{K(n,k)}\leq \frac{\int_M  u P u\, dv_g}{\left(\int_M|u|^{\crit} \, dv_g\right)^{\frac{2}{\crit}}} \hbox{ for all }u\in H_k^2(M)-\{0\}\hbox{ where } P:=\Delta_g^k+B'_0
\end{equation}
We write $P$ of the form \eqref{def:P} with $T\equiv 0$ and $h\equiv 0$ when $k>1$. 
%so that [xxx]$k>2$ et dire que pour k=1 c'est different[xxx]
%$$T\equiv 0\hbox{ and }\left\{\begin{array}{cc}
%h\equiv B'_0&\hbox{ if }k=1\\
%h\equiv 0&\hbox{ if }k>1
%\end{array}\right.$$
Therefore, the test-function estimate \eqref{final:upu} yield for $n\geq 2k+2$
\begin{equation}\label{exp:2}\frac{\int_M \ue P\ue\, dv_g}{\left(\int_M|\ue|^{\crit} \, dv_g\right)^{\frac{2}{\crit}}} =\frac{1}{K(n,k)}-D_{n,k}c_{n,k}R_g(x_0)\theta_\eps+o(\theta_\eps)\hbox{ when }k>1.
\end{equation}
%\begin{eqnarray}
%&&\frac{\int_M \ue P\ue\, dv_g}{\left(\int_M|\ue|^{\crit} \, dv_g\right)^{\frac{2}{\crit}}} =\frac{1}{K(n,k)}+D_{n,k}\left\{\begin{array}{cc}
%B'_0-c_{n,k}R_g(x_0)&\hbox{ if }k=1\\
%-c_{n,k}R_g(x_0)&\hbox{ if }k>1
%\end{array}\right\}\theta_\eps+o(\theta_\eps).
%\end{eqnarray}
Since $R_g(x_0)>0$, we then get a contradiction with \eqref{mino:K} when $n\geq 2k+2$ and $k>1$. Then, \eqref{ineq:opt} cannot hold. 
%Note that when $k=1$, the operator $P$ is written with $T\equiv 0$ and $h\equiv B'_0$, which changes the expansion.
\subsection{The case of Ricci-flat manifolds}We now assume that $k\geq 3$, $n\geq 2k+4$, and $(M,g)$ is Ricci-flat but non flat, that is there exists $x_0\in M$ such that $Rm_g(x_0)\neq 0$. In particular, the Weyl tensor $W_g$ does not vanish at $x_0$. 
%Assume that there exists $B_0'>0$ such that \eqref{ineq:opt} holds, that is 
%\[  \frac{\int_M u (\Dg^k u + B_0' u)\, dv_g}{\left(\int_M |u|^{\crit}\, dv_g\right)^{\frac 2\crit}} \geq \frac1{K(n,k)} \quad \text{for all } u \in H_k^2 \setminus \{0\}.
%\]
Since $Ric_g \equiv 0$, then $(M,g)$ is Einstein and $P_g^k = \Dg^k$ (see Fefferman-Graham \cite{fefgram}). The conformal change of metric $g_{x_0} = \varphi_{x_0}^{4/(n-2k)}g$ yields
\[  \int_M \ue (\Dg^k \ue + B_0' \ue)\,dv_g = \int_M \hue P_{g_{x_0}}^k \hue\, dv_{g_{x_0}}+ B_0' \int_M \varphi_{x_0}^{2-\crit}\hue^2 \, dv_{g_{x_0}}.
\]
A direct application of Proposition 2.1 in \cite{mv} yields $C=C(n,k)>0$ such that
\begin{equation}\label{eq:compMV}
    \frac{\int_M \hue P_{g_{x_0}}^k \hue\, dv_{g_{x_0}}}{\left(\int_M|\hue|^\crit\,dv_{g_{x_0}}\right)^{\frac2\crit}} = \frac 1{K(n,k)} \left(1 - C\epsilon^4 \begin{cases}
    |W_{g_{x_0}}(x_0)|^2 \ln\frac1\epsilon + O(1) & \text{if } n=2k+4\\
    |W_{g_{x_0}}(x_0)|^2 + o(1) & \text{if } n>2k+4
\end{cases}\right)
\end{equation}
as $\epsilon\to 0$. Using the fact that $W_g$ is conformally invariant, we get that $|W_{g_{x_0}}(x_0)|\neq 0$ by hypothesis. Straightforward computations now give 
\begin{equation}\label{eq:estimhue}  
         \int_M \varphi_{x_0}^{2-\crit} \hue^2\, dv_{g_{x_0}} = \int_{B_\delta(0)} (\hue\circ\exp_{x_0}^{g_{x_0}})^2 (1+O(|x|^2)) dx + O(\epsilon^{n-2k})     = o(\epsilon^4)
 \end{equation}
since $k\geq 3$ and $n\geq 2k+4$. Therefore, putting \eqref{eq:compMV} and \eqref{eq:estimhue} together, for any $B_0'\in\rr$, the test-function estimate becomes 
\begin{multline*}
    \frac{\int_M u_\eps (\Dg^k u_\eps + B_0' u_\eps)\, dv_g}{\left(\int_M |u_\eps|^{\crit}\, dv_g\right)^{\frac 2\crit}} = \frac 1{K(n,k)}\\
   \Bigg(1 - C\epsilon^4   \times\begin{cases}
        |W_{g_{x_0}}(x_0)|^2 \ln\frac1\epsilon + O(1) & \text{if } n=2k+4\\
        |W_{g_{x_0}}(x_0)|^2 + o(1) & \text{if } n>2k+4
    \end{cases}\Bigg)
\end{multline*}
which gives the contradiction as in the preceding case. Therefore \eqref{ineq:opt} does not hold when $(M,g)$ is Ricci-flat, but non flat, and $n\geq 2k+4$, $k\geq 3$.

\medskip\noindent The non-validity assertions of Theorems \ref{th:ok:Iopt:bis} and  \ref{th:ok:noIopt:bis} are consequences of these test-functions estimates.

\section{Proof of Theorem \ref{th:bndua:der} assuming Theorem \ref{prop:unicC0} }\label{sec:part4}
We adopt the notations of Theorem \ref{th:bndua:der}.  Let $(u_\a)_\a \in C^{2k}(M)$ be a sequence of  solutions $u_\a \not\equiv 0$ to \eqref{eq:ua:th22} satisfying \eqref{bnd:nrj}.
%We adopt the notations of Theorem \ref{th:bndua:der}.  Let $(u_\a)_\a \in C^{2k}(M)$ be a sequence of  solutions $u_\a \not\equiv 0$ to 
%  \begin{equation}\label{eq:ua:th22} \Dg^k u_\a + \a^{2k} u_\a = |u_\a|^{\crit-2} u_\a \qquad \text{in } M, \quad \forall~\a >0,
%  \end{equation}
%    satisfying 
%\begin{equation}\label{bnd:nrj}
% \int_M{|u_\a|^\crit} \leq K(n,k)^{-\frac{n}{2k}}+o(1).
% \end{equation} 
 \begin{lemma}
    There exists $x_\a \in M$ and $\mu_\a >0$ such that $\a \mu_\a \to 0$ as $\a \to \infty$, and 
    \[  \mu_\a^{\frac{n-2k}{2}} u_\a(\exp_{x_\a}^g(\mu_\a \cdot)) \to \eBub \qquad \text{in } C^{2k}_{loc}(\rr^n).
        \]
    Moreover, $\lim_{\a \to \infty} \Lnorm[(M)]{\crit}{u_\a}^\crit = K(n,k)^{-\frac{n}{2k}}$.
\end{lemma}
\begin{proof}
    Since $u_\a \not \equiv 0$, we let $\bar{x}_\a\in M$ be the point where $|u_\a|$ reaches its maximum, and  $ \bmu_\a := |u_\a(\bar{x}_\a)|^{-\frac{2}{n-2k}}$. Testing \eqref{eq:ua:th22} against $u_\a$, we obtain
    \begin{equation}\label{eq:testua}
        \intM{\bpr{(\Dg^{k/2}u_\a)^2 + \a^{2k} \ua^2}} = \intM{|u_\a|^{\crit}} \leq K(n,k)^{-\frac{n}{2k}}+o(1),
        \end{equation}
    so that by Hölder's inequality,
    \[  \Lnorm[(M)]{\crit}{u_\a}^{\crit} \leq \Lnorm[(M)]{2}{u_\a}^2 \Lnorm[(M)]{\infty}{u_\a}^{\crit-2} \leq \frac{\Lnorm[(M)]{\crit}{u_\a}^{\crit}}{\a^{2k}\bmu_\a^{2k}}.
        \]
    This shows that $\a\bmu_\a$ remains bounded as $\a \to \infty$. Let us define 
    \[  v_\a(y) := \bmu_\a^{\frac{n-2k}{2}} u_\a(\exp_{\bar{x}_\a}^g(\bmu_\a y)) \qquad \forall~y \in \Bal{0}{\delta/\bmu_\a} \sub \rr^n,
        \]
    and we set $\Tg_\a := \exp_{\bar{x}_\a}^* g(\bmu_\a \cdot\,)$. Observe that $|v_\a| \leq 1$ uniformly in $\a$, and   $v_\a$ solves the equation $ (\Dg[\Tg_\a]^k + \a^{2k} \bmu_\a^{2k})v_\a = |v_\a|^{\crit-2} v_\a $ in $ \Bal{0}{\delta/\bmu_\a}$. Thus, since $\Tg_\a \to \xi$ the Euclidean metric in $C^\infty_{loc}(\rr^n)$, $(\Dg[\Tg_\a]^k + \a^{2k}\bmu_\a^{2k})$ is an elliptic operator with bounded coefficients. By standard elliptic theory (see Agmon-Douglis-Nirenberg \cite{ADN} or the Appendix in \cite{robert:gjms}) there exists $v_\infty \in C^{2k}(\rr^n)$ such that $v_\a \to v_\infty$ in $C^{2k}_{loc}(\rr^n)$. Then we have $|v_\infty(0)| =1$, and $v_\infty \not\equiv 0$. Thanks to \eqref{eq:testua}, we also see that $u_\a \to 0$ in $L^2(M)$, so that $u_\a\rightharpoonup 0$ in $\Sob(M)$. Using \eqref{eq:epsineq}, we then obtain that for all $\epsilon >0$, 
    
%    \begin{equation}\label{tmp:bnduadeu}\begin{aligned}  
%        \Lnorm[(M)]{\crit}{u_\a}^2 &\leq (K(n,k)+\epsilon) \Lnorm[(M)]{2}{\Dg^{k/2}u_\a}^2 + B_\epsilon \Lnorm[(M)]{2}{u_\a}^2\\
%            &\leq (K(n,k) + \epsilon)(K(n,k)^{-\frac{n}{2k}} + o(1)) + o(1)\\
%            &\leq K(n,k)^{-\frac{n-2k}{2k}} + C\epsilon.
%    \end{aligned}\end{equation}
%    It then follows that for all $\epsilon>0$,
    \begin{align*}
        \intM{\bpr{(\Dg^{k/2}u_\a)^2 + \a^{2k} \ua^2}} &\leq (K(n,k)^{-\frac{n}{2k}}+o(1))^{1-\frac{2}{\crit}} \bpr{\intM{|u_\a|^\crit}}^{\frac{2}{\crit}}\\
            &\leq K(n,k)^{-1}(K(n,k) + \epsilon)\int_M(\Delta_g^{k/2}u_\alpha)^2\, dv_g +o(1),
    \end{align*}
    so that $\lim_{\a\to \infty} \a^{2k}\intM{u_\a^2} = 0$. Finally, we see that for $R>0$,
    \begin{eqnarray*}
        \a^{2k}\bmu_\a^{2k}\bpr{\int_{\Bal{0}{R}} v_\infty^2 \, dy +o(1)}= \a^{2k}\intM[\Bal{\bar{x}_\a}{R\bmu_\a}]{u_\a^2} \leq o(1).
    \end{eqnarray*}
    Since $v_\infty \not\equiv 0$, we have that $\a\bmu_\a \to 0$. Therefore, $v_\infty \in \hSob(\rn)$ solves the limiting equation  $ \Dg[\xi]^k v_\infty = |v_\infty|^{\crit-2} v_\infty $  in $ \rr^n$, and thus $\Lnorm[(\rr^n)]{\crit}{v_\infty}^{\crit} \geq K(n,k)^{-\frac{n}{2k}}$. We now compute, for all $R>0$,
    \[  \intM{|u_\a|^{\crit}} \geq \intM[\Bal{\bar{x}_\a}{R\bmu_\a}]{|u_\a|^\crit} = \int_{\Bal{0}{R}} |v_\infty|^\crit dy + o(1),
        \]
    so that together with \eqref{bnd:nrj}, we have $\lim_{\a \to \infty} \Lnorm[(M)]{\crit}{u_\a}^\crit = K(n,k)^{-\frac{n}{2k}}$.     This also implies that $\Lnorm[(\rr^n)]{\crit}{v_\infty}^\crit = K(n,k)^{-\frac{n}{2k}}$. Therefore $v_\infty$ is an extremal for the Sobolev inequality \eqref{def:K}. Thus, by Swanson \cite{swanson} and Wei-Xu \cite{WeiXu99},  we conclude that $v_\infty = \eBub$ the positive Euclidean bubble. 
\end{proof}
\begin{proof}[Proof of Theorem \ref{th:bndua:der}] Arguing as in Carletti \cite{Car24}, we get that
%
%    Using the local convergence of $v_\a$ to $\eBub$, we obtain that for all $R>0$,
%    \begin{align*}  
%        \lim_{\a\to \infty} \intM[M\setminus \Bal{\bar{x}_\a}{R\bmu_\a}]{|u_\a|^\crit} &= \lim_{\a\to \infty} \intM{|u_\a|^\crit} - \lim_{\a\to \infty} \intM[\Bal{\bar{x}_\a}{R\bmu_\a}]{|u_\a|^\crit}\\
%            &= K(n,k)^{-\frac{n}{2k}} - \int_{\Bal{0}{R}} \eBub^\crit \, dy = \epsilon(R),
%        \end{align*}
%    where $\epsilon(R) \to 0$ as $R\to \infty$. Therefore, writing $\bar{\nu}_\a := (\bar{x}_\a,\bmu_\a)$, we have that for all $R>0$,
%    \begin{multline*}
%        \lim_{\a\to \infty} \intM{|u_\a - V_{\alpha,\bar{\nu}_\alpha}|^\crit}\\
%        \begin{aligned}
%            &= \lim_{\a\to \infty}\bpr{\intM[\Bal{\bar{x}_\a}{R\bmu_\a}]{|u_\a - V_{\alpha,\bar{\nu}_\alpha}|^\crit} + \intM[M\setminus \Bal{\bar{x}_\a}{R\bmu_\a}]{|u_\a-V_{\alpha,\bar{\nu}_\alpha}|^\crit}}\\
%            &= \lim_{\a\to \infty} \intMg{\Tg_\a}{\Bal{0}{R}}{\abs{\bmu_\a^{\frac{n-2k}{2}}u_\a(\exp_{\bar{x}_\a}(\bmu_\a \cdot\,)) - \eBub}^\crit} + \epsilon(R)= \epsilon(R),
%        \end{aligned} 
%    \end{multline*}
%    where $\TBuba$ is as defined in \eqref{def:TBub}, and writing $\Tg_\a = \exp_{\bar{x}_\a}^*g (\bmu_\a \cdot\,)$ as before. Letting $R\to \infty$, we have showed that 
$u_\a - V_{\alpha,\bar{\nu}_\alpha} \to 0 $ in $L^\crit(M)$, and thus$\Snorm{k}{u_\a-V_{\alpha,\bar{\nu}_\alpha}}^2 \leq \intM{(\Dg^k + \a^{2k})(u_\a-V_{\alpha,\bar{\nu}_\alpha})\, (u_\a - V_{\alpha,\bar{\nu}_\alpha})} =o(1)$. Define $\nu_\alpha:=(x_\a, \mu_\a)$ be such that $\Snorm{k}{u_\a - V_{\alpha, \nu_\alpha}}=\min \Snorm{k}{u_\a - \TBub}$ where the infimum is taken for $\pa = (z,\mu) \in \parama[\a\mu_\a]$.
%    as the solutions to the minimizing problem 
%    \[  minimize~\Snorm{k}{u_\a - \TBub} \qquad \text{for $\pa = (z,\mu) \in \parama[\a\mu_\a]$}.
%        \]
    We then have as in \cite[Lemma 2.9]{Car24} that $\frac{\mu_\a}{\bmu_\a} \to 1$ and  $ \frac{\dg{x_\alpha,\bar{x}_\a}}{\mu_\a} \to 0$ as $\a \to \infty$.  Moreover, $u_\a - V_{\alpha,\nu_\alpha} \to 0$ in $\Sob(M)$, and $u_\a - V_{\alpha,\nu_\alpha}  \in \Kera[\nu_\alpha]^\perp$ defined in \eqref{def:ker} below. By Theorem \ref{prop:unicC0}, there is a unique function $\Tp_\a \in \Kera[\nu_\alpha]^\perp$ satisfying $|\Tp_\a(x)| \leq \Lambda\, \Teta[\bpa_\a](x) \Bub_{x_\alpha,\mu_\a}(x)$ for all $x\in M$, which is solution to \eqref{eq:modcritC0}. Writing $\phi_\a := u_\a - \TBuba[\nu_\alpha]$, it follows from \eqref{eq:ua:th22} that $\phi_\a\in \Kera[\nu_\alpha]^\perp$ solves 
    \[  \proKa[\nu_\alpha]\bpr{\TBuba[\nu_\alpha] + \phi_\a - (\Dg^k +\a^{2k})^{-1} \bpr{|\TBuba[\nu_\alpha] + \phi_\a|^{\crit-2}(\TBuba[\nu_\alpha] + \phi_\a)}} = 0.
        \] 
    We have $\Snorm{k}{\phi_\a} \to 0$ and $\Vert\Tp_\a\Vert_{H_k^2(M)} \to 0$ as $\a \to \infty$, so that the uniqueness of Proposition \ref{prop:unicHk}  allows to conclude that $\phi_\a = \Tp_\a$ for all $\a>0$ such that $\a\mu_\a < \rho_0$ and $\Snorm{k}{\phi_\a}< R_0$ and $\Vert\Tp_\a\Vert_{H_k^2(M)} <R_0$. Finally, for any $\tau\in (0,2]$ such that $0<\tau\leq \min\left\{2,\frac{n-2k}{2}\right\}$, noting that $\tau<2k$ since $k>1$, estimate \eqref{bnd:phi} of Theorem \ref{prop:unicC0} yields $u_\a = \TBuba[\nu_\alpha] + \Tp_\a$ satisfies
    \begin{align}\label{ineq:inter}  
        |\nabla_g^l \left(u_\a -\TBuba[\nu_\alpha]\right)(x)| &\leq C\frac{(\alpha(\ma+d_g(\xa, x)))^\tau}{(\mu_\a+ \dg{x_\a,x})^{l}} \Bub_{x_\alpha,\mu_\alpha}(x)
        \end{align}
    for all $x\in M$, and  $l=0,\ldots, 2k-1$, which proves Theorem \ref{th:bndua:der}. 
    \end{proof}
As already mentioned, Theorem \ref{th:bndua} is a consequence of Theorem \ref{th:bndua:der}. Just note that the upper bound \eqref{eq:estimbndua} is a consequence of \eqref{ineq:inter} and the explicit expression of $\TBuba$.

\section{Theorems \ref{th:ok:Iopt:bis} and  \ref{th:ok:noIopt:bis}: Proof of the validity part via blow-up analysis}\label{sec:blowup}
%The goal is to obtain a contradiction with the existence of the blowing-up sequence $u_\a$ as $\a \to \infty$. This is achieved through a Pohozaev identity. The higher order term will bring up a term $R_g(x_\a)\mu_\a^2$, while the lower order term $\int_M u_\a^2$ is estimated in \eqref{est:udeux} and behaves as $-(\a\mu_\a)^{2k}$ up to a multiplicative constant when $n>4k$. This will allow to reach a contradiction and conclude to the validity of \eqref{ineq:opt} in all the cases considered in Theorems \ref{th:ok:Iopt:bis} and  \ref{th:ok:noIopt:bis}.
%
We argue by contradiction and we assume that \eqref{ineq:opt} does not hold and that $k>1$. Then for all $\alpha>0$, there exists $w\in H_k^2(M)-\{0\}$ such that 
\begin{equation*}
J_\alpha(w):= \frac{\int_M(\Delta_g^{k/2}w)^2\, dv_g +\alpha^{2k}\int_Mw^2\, dv_g}{\left(\int_M|w|^{\crit}\, dv_g\right)^{\frac{2}{\crit}}} <\frac{1}{K(n,k)}.
\end{equation*}
It follows from Mazumdar \cite{mazumdar:jde} that there exists $\ua\in C^{2k}(M)$ such that 
$$J_\alpha(\ua)=\inf\{J_\alpha(u)/\, u\in H_k^2(M)-\{0\}\}:=\lambda_\alpha<\frac{1}{K(n,k)}.$$
Up to multiplying $\ua$ by a suitable constant, we can assume that 
\begin{equation}\label{eq:ua}
\Delta_g^k\ua+\alpha^{2k}\ua=|\ua|^{\crit-2}\ua\hbox{ in }M
\end{equation}
so that  $ \int_M|\ua|^{\crit}\, dv_g = \lambda_\alpha^{\frac{n}{2k}}<\frac{1}{K(n,k)^{\frac{n}{2k}}}.$

The goal is to obtain a contradiction with the existence of this sequence $u_\a$ as $\a\to \infty$. This is achieved through a Pohozaev inequality, where each term is estimated precisely. The imbalance between each terms' rate of decay will allow to reach the contradiction and conclude to the validity of \eqref{ineq:opt} in all the cases considered in Theorems \ref{th:ok:Iopt:bis} and \ref{th:ok:noIopt:bis}.

\medskip\noindent   Theorem  \ref{th:bndua:der} yields $(\xa)_\alpha\in M$ and $(\ma)_{\alpha>0}\in \rr_{>0}$ such that 
\begin{equation}\label{lim:hua}
\lim_{\alpha\to +\infty}\ma=0\, ;\, \lim_{\alpha\to +\infty}\alpha \ma =0\, ;\, \lim_{\alpha\to +\infty}\ma^{\frac{n-2k}{2}}\ua(\hbox{exp}_{\xa}(\ma\cdot))=U\hbox{ in }C^{2k}_{loc}(\rn)
\end{equation}
where $U$ is as in \eqref{def:U}. %Moreover, for all $q>1$, there exists $C_q>0$ such that for all $l=0,...,2k-1$, we have that
%\begin{equation}\label{ineq:cl}
%|\nabla_g^l\ua(x)|\leq \frac{C_q}{ \big(1+ \alpha d_g(x,\xa)\big)^{q}}\cdot\frac{\ma^{\frac{n-2k}{2}}}{(\ma+d_g(x,\xa))^{n-2k+l}}\hbox{ for all }x\in M\hbox{ and }\alpha>0.
%\end{equation}
As one checks, we have that $\lim_{\alpha\to +\infty}\xa=x_\infty\in M$. We use the notations of \cite{robert:gjms} and of \eqref{def:phi} of Section \ref{sec:test} above. We set $\phia:=\varphi_{\xa}\circ \hbox{exp}^{g_{x_\alpha}}_{\xa}$ and $g_\alpha:= (\hbox{exp}^{g_{x_\alpha}}_{\xa})^\star g_{\xa}$.  We rewrite equation \eqref{eq:ua} with the GJMS operator \eqref{def:gjms} as
\begin{equation*}
P_g^k\ua+\sum_{i=1}^{k-1} (-1)^i\nabla_g^i((-A_{g}^{(i)}) \nabla_g^i\ua)+ (\alpha^{2k}-A_g^{(0)})\ua=|\ua|^{\crit-2 }\ua\hbox{ in }M.
\end{equation*}
We define 
\begin{equation}\label{def:hua}
\hua:=(\varphi_{x_\alpha}^{-1}\ua)\circ \tilde{\hbox{exp}}_{\xa}=\phia^{-1}\cdot \ua\circ \tilde{\hbox{exp}}_{\xa} : B_\delta(0)\subset\rn\to \rr.
\end{equation}
Using the conformal properties \eqref{invar:gjms} of the GJMS operator, we get that $P_{g_{\xa}}(\varphi_{\xa}^{-1}\ua)=\varphi_{\xa}^{-(\crit-1)}P_g^k\ua$. Plugging this expression in \eqref{eq:ua} yields
\begin{equation}\label{eq:hua}
\Delta_{\ga}^k\hua+\sum_{i=0}^{k-1} (-1)^i\nabla^i_{\ga}(\hat{A}_\alpha^{(i)}  \nabla^i_{\ga}\hua)+\alpha^{2k}\phia^{2-\crit}\hua= |\hua|^{\crit-2 }\hua\hbox{ in }B_\delta(0).
\end{equation}
where $\hat{A}^{(i)}_\alpha$ is a smooth symmetric $(0, 2i)-$tensor for all $i=0,...,k-1$ and there exists $\Vert \hat{A}^{(i)}_\alpha\Vert_{C^{i,\theta}}\leq C$ for all $\alpha$. 
%Moreover, see Subsection 15.2 in \cite{robert:gjms}
%\begin{eqnarray}\label{eq:76bis}
%(-1)^{k-1}\nabla^{k-1}_{\ga}(\hat{A}_\alpha^{(k-1)}  \nabla^{k-1}_{\ga}\hua)&=& -\Delta_{\ga}^{k-2}(\hat{T}_\alpha,\nabla_{\ga}^2\hua)+lot
%\end{eqnarray}
Arguing as in \cite{robert:gjms}, Section 15.2, and using formula (2.7) in \cite{mv}, we get that 
\begin{eqnarray}\label{eq:76bis}%{eq:lotinAk-1}
    (-1)^{k-1}\nabla^{k-1}_{\ga}(\hat{A}_\alpha^{(k-1)}  \nabla^{k-1}_{\ga}\hat{u}_\a)&=& -\Delta_{\ga}^{k-2}(\hat{T}_\alpha,\nabla_{\ga}^2 \hat{u}_\a) - \D_{\ga}^{k-3}(\hat{T}^{(3)}_\a, \nabla_{\ga}^3 \hat{u}_\a)\nonumber\\
        && + O\left(|x||\nabla^{2k-3} \hat{u}_\a| + \sum_{m=0}^{2k-4} |\nabla^m \hat{u}_\a|\right).
    \end{eqnarray}
    where
    \begin{equation}\label{def:hat:T}
\hat{T}_\alpha:=\phia ^{-\frac{8}{n-2k}} (\tilde{\hbox{exp}}_{\xa}^{g_{\xa}})^\star(-T_g)+ (\tilde{\hbox{exp}}_{\xa}^{g_{\xa}})^\star T_{\tga}\hbox{ when }k>1
\end{equation}
and $T_g$ is as in \eqref{def:T:g}. In the expression above, $\hat{T}^{(3)}_\a$ is a smooth symmetric $(0,3)$-tensor. The Pohozaev-Pucci-Serrin identity (as stated for instance in \cite{robert:gjms}) reads  
\begin{multline*}
\int_{B_\delta(0)}\left(\Delta_\xi^k \hua- |\hua|^{\crit-2}\hua\right) T(\hua)\, dx\\
=   \int_{\partial B_\delta(0)}\left((x,\nu)\left(\frac{(\Delta_\xi^\frac{k}{2}\hua)^2}{2}-\frac{|\hua|^{\crit }}{\crit }\right)+S(\hua)\right)\, d\sigma
\end{multline*}
where $T(\hua):=\frac{n-2k}{2}u+x^i\partial_i u$ and
\begin{eqnarray*}
S(\hua)&:=&\sum_{i=0}^{E(k/2)-1}\left(-\partial_\nu\Delta_\xi^{k-i-1}\hua\Delta_\xi^iT(\hua)+\Delta_\xi^{k-i-1}\hua\partial_\nu\Delta_\xi^iT(\hua)\right)\\
&&-{\bf 1}_{\{k\hbox{ odd}\}}\partial_\nu(\Delta_\xi^\frac{k-1}{2} \hua)\Delta_\xi^\frac{k-1}{2}T(\hua)
\end{eqnarray*}
Using \eqref{eq:hua}, we get $III_\alpha-II_\alpha= IV_\alpha$ where
\begin{equation*}%\label{def:IIalpha}
II_\alpha:=\int_{B_{\delta}(0)}\left(\sum_{i=0}^{k-1} (-1)^i\nabla^i_{\ga}(\hat{A}_\alpha^{(i)}  \nabla^i_{\ga}\hua) \right) T(\hua)\, dx +\alpha^{2k}\int_{B_{\delta}(0)}   \phia^{2-\crit}   \hua  T(\hua)\, dx
\end{equation*}
\begin{equation*}%\label{def:IIIalpha}
III_\alpha:=\int_{B_{\delta}(0)}\left(\Delta_\xi^k\hua  -\Delta_{\ga}^k  \hua\right) T(\hua)\, dx
\end{equation*}
\begin{equation*}%\label{def:IValpha}
IV_\alpha:=\int_{\partial B_{\delta}(0)}\left((x,\nu)\left(\frac{(\Delta_\xi^\frac{k}{2}\hua)^2}{2}-\frac{ |\hua|^{\crit }}{\crit }\right)+S(\hua)\right)\, d\sigma
\end{equation*}
We estimate these terms separately. First note that for $l=0,...,2k-1$, using the explicit expression of $\TBuba$, the pointwise control \eqref{eq:estimbndua:der} rewrites
\begin{equation}\label{ineq:59}
|\nabla^l_\xi\hua(x)|\leq \frac{C_q}{ \big(1+ \alpha |x|\big)^{q}}\cdot \frac{\ma^{\frac{n-2k}{2}}}{\left(\ma+|x|\right)^{n-2k+l}}\hbox{ for all }x\in B_{\delta}(0).
\end{equation}
%Using \eqref{ineq:59}, we get that
%\begin{equation}\label{ineq:60}
%\int_{B_{\da}(0)}|x|^t|\nabla^l \hua|\left(|x|\cdot|\nabla \hua|+|\hua|\right)\, dx\leq C\left\{\begin{array}{cc}
%\ma^{t+2k-l} &\hbox{ if }n-4k+l>t\\
%\ma^{n-2k}\ln\frac{1}{\ma} &\hbox{ if }n-4k+l=t\\
% \ma^{n-2k}  &\hbox{ if }n-4k+l<t
%\end{array}\right.
%\end{equation}
We define $\va(x):=\ma^{\frac{n-2k}{2}}\hua( \ma x)=\ma^{\frac{n-2k}{2}}\phia^{-1}(\ma x)\cdot\ua\circ  \hbox{exp}^{g_\alpha}_{\xa}( \ma x))$ for all $x\in B_{\delta/\ma}(0)$. Since $\phia(0)=1$, \eqref{lim:hua} writes
\begin{equation}\label{lim:tua:c2k:bis}
\lim_{\alpha\to +\infty}\va=U\hbox{ in }C^{2k}_{loc}(\rn).
\end{equation}

\subsection{Estimates of $III_\a$ and $IV_\a$.}
% In the sequel, we are taking $\tau\in \rr$ such that 
%\begin{equation}\label{choice:tau}
%\tau\in (0,2],\, \tau<2k\hbox{ and }\tau \leq \frac{n-2k}{2}.
%\end{equation}
Using the $C^{2k-1}-$pointwise control of Theorem \ref{th:bndua} for $q$ large enough, we get that $ IV_\alpha=O\left(\frac{\ma^{n-2k}}{\alpha^{2q}}\right)$. Note that $u_\a-\TBuba=\phi_\a$. We define
\[\begin{array}{ccc}
    \hat{V}_{\a} := (\varphi_{x_\a}^{-1} \TBuba)\circ\hbox{exp}^{g_{x_\alpha}}_{x_\a} & \hbox{and} &\hat{\phi}_\a := (\varphi_{x_\a}^{-1} \phi_\a)\circ\hbox{exp}^{g_{x_\alpha}}_{x_\a},
\end{array}\]
for $\pa_\a = (x_\a, \mu_\a)$, where $\varphi_{x_\a}$ is as in \eqref{def:phi} and $\TBuba$  as in \eqref{def:TBub}. Therefore the pointwise control \eqref{eq:estimbndua:der} of Theorem \ref{th:bndua:der} writes
 \begin{equation}\label{bnd:phi:hat}  (\ma + |x|)^l |\nabla_\xi^l \hat{\phi}_\alpha(x)| \leq C\, \hat{\Theta}_\alpha(x) U_{\ma}(x)\hbox{ for all }x\in B_\delta(0)
 \end{equation}
for $l=0,\ldots, 2k-1$, where
 \begin{equation}\label{def:Umu}
\hat{\Theta}_\alpha(x):=\frac{(\alpha(\ma+|x|))^\tau}{(1+\alpha|x|)^{ p}}\hbox{ and }U_{\ma}(x):=\left(\frac{\ma}{\ma^2+a_{n,k}|x|^2}\right)^{\frac{n-2k}{2}} ,
\end{equation}
 for all $x\in B_{\delta}(0)$ and $\alpha>1$. Since  $\hat{V}_{\a} (x)=\chi(\alpha |x|)U_{\ma}(x)$ for all $x\in B_\delta(0)$, we get
\begin{equation}\label{bnd:V:hat}  (\ma + |x|)^l |\nabla_\xi^l \hat{V}_{\a}(x)| \leq C{\bf 1}_{\alpha|x|\leq 2}\, U_{\ma}(x)\hbox{ for all }x\in B_\delta(0)
 \end{equation}
We define the linear form $Q_\alpha:=\Delta_\xi^k - \Delta_{\ga}^k$. Since the Ricci tensor  of $g_{x_\alpha}$ vanishes at $x_\alpha$, an expansion yields (see for instance \cite{robert:gjms})
\begin{eqnarray*}
Q_\alpha(u)&=&(x)^2 \star\nabla_\xi^{2k}u+O\left(|x|\cdot|\nabla_\xi^{2k-1}u|+\sum_{i\leq 2k-2}|\nabla_\xi^{i}u|\right)
\end{eqnarray*}
for all $u\in C^{2k}(B_{2\delta}(0))$, where $(x)^q$ is the product of a homogeneous polynomial of degree $q$ by a smooth  function or tensor and  $T\star S$ denotes any linear combination and contraction of two tensors $T$ and $S$. The decomposition $\hua = \hat{V}_{\a} + \hat{\phi}_\a$ yields
\begin{eqnarray*}
 III_\alpha&=& \int_{B_\delta(0)}Q_\alpha(\hua) T(\hua) \,dx=\int_{B_\delta(0)}Q_\alpha( \hat{V}_{\a} + \hat{\phi}_\a)  T( \hat{V}_{\a} + \hat{\phi}_\a) \,dx\\
 &=&  \int_{B_\delta(0)}Q_\alpha( \hat{V}_{\a} )  T( \hat{V}_{\a}  ) \, dx+ \int_{B_\delta(0)}Q_\alpha( \hat{V}_{\a} )  T(\hat{\phi}_\a)\, dx+ \int_{B_\delta(0)}Q_\alpha( \hat{\phi}_\a )  T(\hua)\, dx
\end{eqnarray*}
We deal with the three terms separately. The control \eqref{bnd:V:hat} yields
$$Q_\alpha( \hat{V}_{\a} )(x)=O\left({\bf 1}_{\alpha |x|<2}\frac{U_{\ma}(x)}{(\ma+|x|)^{2k-2}}\right)$$
for all $x\in B_\delta(0)$. The control \eqref{bnd:phi:hat}   yields $T(\hat{\phi}_\a)=O\left(\hat{\Theta}_\alpha U_{\ma}\right)$.
%where
%$$\hat{\Theta}_\alpha(x):=\frac{(\alpha(\ma+|x|))^\tau}{(1+\alpha|x|)^{2kp}}\hbox{ for all }x\in B_{\delta}(0)\hbox{ and }\alpha>1.$$
We then get that 
$$\int_{B_\delta(0)}Q_\alpha( \hat{V}_{\a} )  T(\hat{\phi}_\a)\, dx=O\left(\int_{B_\delta(0)}{\bf 1}_{\alpha |x|<2}\frac{\hat{\Theta}_\alpha U_{\ma}^2}{(\ma+|\cdot|)^{2k-2}}\, dx\right).$$
Using \eqref{bnd:phi:hat} for derivatives up to $2k-1$, the same strategy yields
\begin{multline*}
\int_{B_\delta(0)}Q_\alpha( \hat{\phi}_\a )  T(\hua)\, dx\\
=\int_{B_\delta(0)}((x)^2 \star\nabla_\xi^{2k}\hat{\phi}_\alpha )\cdot T(\hua)\, dx+O\left(\int_{B_\delta(0)} \frac{\hat{\Theta}_\alpha U_{\ma}^2}{(\ma+|\cdot|)^{2k-2}}\, dx\right)
\end{multline*}
Integrating by parts allows to recover derivatives of $\phi_\a$ of order at most $2k-1$, since $k>1$. The same computations as above then yield 
%\begin{equation*}
%\int_{B_\delta(0)}((x)^2 \star\nabla_\xi^{2k}\hat{\phi}_\alpha )\cdot T(\hua)\, dx = -\int_{B_\delta(0)}\nabla((x)^2 T(\hua))\star\nabla_\xi^{2k-1}\hat{\phi}_\alpha \, dx+\int_{\partial B_\delta(0)} (x)^2 T(\hua)\star\nabla_\xi^{2k-1}\hat{\phi}_\alpha \, d\sigma.
%\end{equation*}
%We then recover derivatives of order $2k-1$ at most since $k>1$. The same computations as above yield
\begin{equation*}
\int_{B_\delta(0)}((x)^2 \star\nabla_\xi^{2k}\hat{\phi}_\alpha )\cdot T(\hua)\, dx = O\left(\int_{B_\delta(0)} \hat{\Theta}_\alpha\frac{U_{\ma}^2}{(\ma+|\cdot|)^{2k-2}}\, dx\right)+O\left(\frac{\ma^{n-2k}}{\alpha^{2q}}\right).
\end{equation*}
Now since $\hat{V}_{\a}$ is radially symmetrical, it follows that taking $N$ large enough in \eqref{def:phi}, as in \eqref{diff:riem:eucl}, we have $\Delta_{\ga}^k \hat{V}_\a = \Delta_\xi^k \hat{V}_\a + O(\mu_\a^{\frac{n-2k}{2}} |x|^{N-n})$. Therefore, since $\hat{V}_{\a}  =\chi(\alpha |\cdot|)U_{\ma}$ has support in  $B_{2\alpha^{-1}}(0)$, the above estimates yield
\begin{align*}%\label{estim:IIIalpha}
    &III_\alpha=\int_{B_\delta(0)}\left(\Delta_\xi^k \hat{V}_\a - \Delta_{\ga}^k \hat{V}_\a\right) T(\hat{V}_\a) dx\\ &\qquad +O\left(\int_{B_\delta(0)} \frac{\hat{\Theta}_\alpha U_{\ma}^2}{(\ma+|\cdot|)^{2k-2}}\, dx\right)+O\left(\frac{\ma^{n-2k}}{\alpha^{2q}}\right)\\
    &=   O\left(\int_{B_{2\alpha^{-1}}(0)} \mu_\a^{\frac{n-2k}{2}} |x|^{N-n} U_{\ma}(x)\,dx  +\int_{B_\delta(0)} \frac{\hat{\Theta}_\alpha U_{\ma}^2}{(\ma+|\cdot|)^{2k-2}}\, dx +\frac{\ma^{n-2k}}{\alpha^{2q}}\right)    %\left\{\begin{array}{cc}
\end{align*}
We now choose $\tau\in \rr$ such that 
\begin{equation}\label{choice:tau}
%\tau\in (0,2],\, \tau<2k\hbox{ and }\tau \leq \frac{n-2k}{2}.
0<\tau\leq \min \left\{2,\frac{n-2k}{2}\right\}.
\end{equation}
Note that this implies $\tau<2k$ since $k>1$. We then get the estimate 
%
%
%
%
%Plugging all these estimates together and using the choice of $\tau>0$ in \eqref{choice:tau}, we get that
%\begin{eqnarray*}%\label{estim:IIIalpha}
%    III_\alpha&=&\int_{B_\delta(0)}\left(\Delta_\xi^k \hat{V}_\a - \Delta_{\ga}^k \hat{V}_\a\right) T(\hat{V}_\a) dx +O\left(\int_{B_\delta(0)} \hat{\Theta}_\alpha\frac{U_{\ma}^2}{(\ma+|\cdot|)^{2k-2}}\, dx\right)+O\left(\frac{\ma^{n-2k}}{\alpha^{2q}}\right)\\
%    &=&   \int_{B_\delta(0)}\left(\Delta_\xi^k \hat{V}_\a - \Delta_{\ga}^k \hat{V}_\a\right) T(\hat{V}_\a) dx +O(\theta_\alpha^\prime)+ O\left(\frac{\ma^{n-2k}}{\alpha^{2q}}\right)
%    %\left\{\begin{array}{cc}
%    \end{eqnarray*}
%Now since $\hat{V}_{\a}$ is radially symmetrical, it follows that taking $N$ large enough in \eqref{def:phi}, as in \eqref{diff:riem:eucl}, we have $\Delta_{\ga}^k \hat{V}_\a = \Delta_\xi^k \hat{V}_\a + O(\mu_\a^{\frac{n-2k}{2}} |x|^{N-n})$ and thus, using that $\hat{V}_{\a}  =\chi(\alpha |\cdot|)U_{\ma}$ has support in  $B_{2\alpha^{-1}}(0)$, we get that
$III_\a = O(\theta_\alpha)+O(\theta_\alpha^\prime)$ where we have set
\begin{equation}
\begin{aligned}
    \theta_\a&=\left\{\begin{array}{cc}
    \ma^4 &\hbox{ if }n> 2k+4\\
    \ma^4\ln\frac{1}{\alpha\ma} &\hbox{ if }n = 2k+4\\
    \frac{\ma^{n-2k}}{\alpha^{2k+4-n}}&\hbox{ if }n< 2k+4
    \end{array}\right\},\\
    \theta_\a^\prime&=\left\{\begin{array}{cc}
    \ma^2 (\alpha\ma)^\tau&\hbox{ if }n> 2k+2+\tau\\
    \ma^2 (\alpha\ma)^\tau \ln\frac{1}{\alpha\ma} &\hbox{ if }n= 2k+2+\tau\\
    \ma^2 (\alpha\ma)^{n-2k-2} &\hbox{ if }n<2k+2+\tau
\end{array}\right\}.\label{def:theta}
\end{aligned}
\end{equation}

%\subsection{Estimates of $II_\a$ for $n\geq2k+4$ and $n=2k+1$.}
%Arguing as in \cite{robert:gjms}, Section 15.2, and using again the pointwise control of Theorem \ref{th:bndua} with $q$ large enough, we get that
%\begin{eqnarray*}
%II_\alpha&=&\left\{\begin{array}{cc}
%\ma^2c_{n,k}R_g(x_\infty)\int_{\rn} (\Delta_{\xi}^{\frac{k-1}{2}}U)^2 \, dx+o(\theta_\a)&\hbox{ if }n>2k+4\\
% O\left( \theta_\a\right) &\hbox{ if }n=2k+1,
%\end{array}\right.\\
%&&+\alpha^{2k}\int_{B_{\delta}(0)}   \phia^{2-\crit}   \hua  T(\hua)\, dx
%\end{eqnarray*}
%where $c_{n,k}$ is as in \eqref{def:petit:c}. 

 \subsection{Estimates of $II_\a$.} With the estimate \eqref{eq:estimbndua}, we get that
 \begin{eqnarray*}
    II_\alpha&=&\int_{B_\delta(0)}(-1)^{k-1}\nabla_{g_\a}^{k-1}(\hat{A}_\a^{(k-1)}\nabla_{g_\a}^{k-1}\hat{u}_\a)~T(\hat{u}_\a)\, dx\\
    && +\alpha^{2k}\int_{B_{\delta}(0)}   \phia^{2-\crit}   \hua  T(\hua)\, dx+O(\theta_\a).
    \end{eqnarray*}
 We use the decomposition of $\hua = \hat{V}_{\a} + \hat{\phi}_\a$, and thanks to the estimates on $\phi_\a$ given by \eqref{bnd:phi:hat}, we have that 
\begin{eqnarray*}
    II_\alpha&=&\int_{B_\delta(0)}(-1)^{k-1}\nabla_{g_\a}^{k-1}(\hat{A}_\a^{(k-1)}\nabla_{g_\a}^{k-1}\hat{V}_\a)~T(\hat{V}_\a)\, dx\\
    && +\alpha^{2k}\int_{B_{\delta}(0)}   \phia^{2-\crit}   \hua  T(\hua)\, dx+O(\theta_\a)+O(\theta_\alpha^\prime).
    \end{eqnarray*}
Using \eqref{eq:76bis} and that $\nabla^p_{g_\alpha}T=\nabla^p_{\xi}T+(x)\star\nabla^{p-1}T+\sum_{m<p-1} O(\nabla^m T)$, we get that
%
%Arguing as in \cite{robert:gjms}, Section 15.2, and using formula (2.7) in \cite{mv}, we get that 
%\begin{eqnarray}\label{eq:lotinAk-1}
%    (-1)^{k-1}\nabla^{k-1}_{\ga}(\hat{A}_\alpha^{(k-1)}  \nabla^{k-1}_{\ga}\hat{V}_\a)&=& -\Delta_{\ga}^{k-2}(\hat{T}_\alpha,\nabla_{\ga}^2 \hat{V}_\a) - \D_{\ga}^{k-3}(\hat{T}^{(3)}_\a, \nabla_{\ga}^3 \hat{V}_\a)\\
%        &&\notag + O\left(|x||\nabla^{2k-3} \hat{V}_\a| + \sum_{m=0}^{2k-4} |\nabla^m \hat{V}_\a|\right).
%    \end{eqnarray}
%Here $\hat{T}_\alpha$ is as defined in \eqref{def:hat:T}, and $\hat{T}^{(3)}_\a$ is a smooth symmetric $(0,3)$-tensor whose exact expression involves only tensor products and derivatives of the metric $g$, the Ricci tensor and the scalar curvature. 
%
\begin{eqnarray*}
&&    \int_{B_\delta(0)}(-1)^{k-1}\nabla_{g_\a}^{k-1}(\hat{A}_\a^{(k-1)}\nabla_{g_\a}^{k-1}\hat{V}_\a)~T(\hat{V}_\a)\, dx\\
   &     &= -\int_{B_\delta(0)} \D_{\xi}^{k-2} (\hat{T}_\a, \nabla^2_\xi \hat{V}_\a) ~T(\hat{V}_\a)\, dx - \int_{B_\delta(0)} \D_{\xi}^{k-3}(\hat{T}^{(3)}_\a, \nabla^3_\xi \hat{V}_\a)~T(\hat{V}_\a)\, dx + O(\theta_\a)\\
        &&= A_\a + B_\a + O(\theta_\a),     
\end{eqnarray*} 
where we let 
\begin{eqnarray*}
  &&  A_\a := -\int_{B_{\delta}(0)} \hat{T}_\a^{ij}(0) \partial_{ij}\D_{\xi}^{k-2} \hat{V}_\a ~T(\hat{V}_\a)\, dx,\\
  &&  B_\a := -\int_{B_\delta(0)} x^p \partial_p \hat{T}_\a^{ij}(0) \partial_{ij}\D_\xi^{k-2}\hat{V}_\a~T(\hat{V}_\a)\, dx\\ 
  &&+ 2\int_{B_\delta(0)} \partial_p \hat{T}_\a^{ij}(0)\partial_{pij}\D_\xi^{k-3}\hat{V}_\a~T(\hat{V}_\a)\, dx - \int_{B_\delta(0)}\hat{T}^{(3)ijl}_\a(0) \partial_{ijl}\D_{\xi}^{k-3}\hat{V}_\a~T(\hat{V}_\a)\, dx.
\end{eqnarray*}
Since $\hat{V}_{\a} =\chi(\alpha |\cdot|)U_{\ma}$ where $U_{\ma}$ is as in \eqref{def:Umu}, then $\hat{V}_\a$ and its successive Euclidean Laplacians are radially symmetrical and then $B_\alpha=0$.
%by its definition and the definition of $\TBuba$ in \eqref{def:TBub}, it is easily obtained that 
%\begin{multline*}  
%    \int_{B_\delta(0)} x^p \partial_p \hat{T}_\a^{ij}(0) \partial_{ij}\D_\xi^{k-2}\hat{V}_\a~T(\hat{V}_\a)\, dx\\ = \mu^3\int_0^{\frac{1}{\a\mu_\a}}dr\, r^n \partial_p \hat{T}_\a^{ij}(0)\left(\frac{n-2k}{2}U + r\partial_r U\right) \int_{\mathbb{S}^{n-1}}d\sigma\, \sigma_p \left(\sigma_i\sigma_j \partial_{rr}\D_\xi^{k-2} U + \frac{\delta_{ij}-\sigma_i\sigma_j}{r}\partial_r \D_\xi^{k-2}U\right) + o(\theta_\a)\\
%    = 0 + o(\theta_\a).
%\end{multline*}
%All the other integrals in $B_\a$ are computed similarly, so that  
%\[ B_\a = o(\theta_\a).\]
%\[ B_\a = 0.\]
Finally, arguing as in \cite{robert:gjms}, Section 15.2, and noting that $\hat{V}_\alpha\in C^\infty(\rn)$ is radially symmetrical and has compact support in $B_\delta(0)$, we have that
\begin{eqnarray*}  
    A_\a &= &-\int_{B_\delta(0)} \hat{T}_\a^{ij}(0)~ \partial_{ij}\Delta_\xi^{k-2}\hat{V}_\a T(\hat{V}_\alpha)\,dx\\
    &= &- \sum_i \hat{T}_\a^{ii}(0) \int_{B_\delta(0)}  \partial_{ii}\Delta_\xi^{k-2}\hat{V}_\a T(\hat{V}_\alpha)\,dx\\
    &=&  \frac{Tr(\hat{T}_\a )(0) }{n}\int_{B_\delta(0)}   \Delta_\xi^{\frac{k-1}{2}}\hat{V}_\a \Delta_\xi^{\frac{k-1}{2}}(T(\hat{V}_\alpha))\,dx
%%&= & \frac{\sum_i \hat{T}_\a^{ii}(0) }{n}\int_{B_\delta(0)}   \Delta_\xi^{k-1}\hat{V}_\a T(\hat{V}_\alpha)\,dx\\
% &= & \frac{Tr(\hat{T}_\a )(0) }{n}\int_{B_\delta(0)}   \Delta_\xi^{\frac{k-1}{2}}\hat{V}_\a \Delta_\xi^{\frac{k-1}{2}}(T(\hat{V}_\alpha))\,dx
 \end{eqnarray*}
It follows from Lemma 13.1 in \cite{robert:gjms} that
\begin{eqnarray*}
\Delta_\xi^{\frac{k-1}{2}}(T(\hat{V}_\alpha))&=& \frac{n-2}{2}\Delta_\xi^{\frac{k-1}{2}} \hat{V}_\alpha+x^p\partial_p\Delta_\xi^{\frac{k-1}{2}}\hat{V}_\alpha
\end{eqnarray*}
and therefore
\begin{eqnarray*} 
A_\a &= &  \frac{Tr(\hat{T}_\a )(0) }{n}\int_{B_\delta(0)} \left(  \frac{n-2}{2}(\Delta_\xi^{\frac{k-1}{2}} \hat{V}_\alpha)^2+\frac{x^p}{2}\partial_p(\Delta_\xi^{\frac{k-1}{2}}\hat{V}_\alpha)^2\right)\,dx\\
&=&-\frac{Tr(\hat{T}_\a )(0) }{n}\int_{B_\delta(0)} \left(   \Delta_\xi^{\frac{k-1}{2}} \hat{V}_\alpha\right)^2\,dx
 \end{eqnarray*}

%    
%    
%    &= -\mu_\a^2 \int_{0}^{\frac{1}{\a\mu_\a}} r^{n-1}\left(\partial_r(\Delta_\xi^{\frac{k-2}{2}}U)\right)^2 \, dr \int_{\mathbb{S}^{n-1}} \hat{T}_\a^{ij}(0) \, \sigma_i \sigma_j\, d\sigma +o(\theta_\a)\\
%    &= - \mu_\a^2 \omega_{n-1} \frac{Tr(\hat{T}_\a(0))}{n} b_{n,k}\left\{\begin{array}{cc}
%        1 & \hbox{ if } n=2k+3\\
%        \ln \frac{1}{\a\mu_\a}+ O(1) & \hbox{ if } n=2k+2
%    \end{array}\right\}+o(\theta_\alpha)
%\end{align*}
%for some explicit constant $b_{n,k}>0$.
With these estimates and the definition \eqref{def:hat:T} of $\hat{T}_\a$, we   get that 
\begin{equation}\label{est:II}
II_\a = c_{n,k}R_g(x_\alpha)\int_{B_\delta(0)} \left(   \Delta_\xi^{\frac{k-1}{2}} \hat{V}_\alpha\right)^2\,dx + \alpha^{2k}\int_{B_{\delta}(0)}   \phia^{2-\crit}   \hua  T(\hua)\, dx + O(\theta_\a)+O(\theta_\a^\prime)
\end{equation}
for $c_{n,k}>0$ defined in \eqref{def:petit:c}. Integrating by parts and using \eqref{ineq:59}, we get that
\begin{eqnarray*}
&&\int_{B_{\delta}(0)}   \phia^{2-\crit}   \hua  T(\hua)\, dx=\int_{B_{\delta}(0)}   \phia^{2-\crit}   \hua  \left(\frac{n-2k}{2}\hua+x^i\partial_i \hua\right)\, dx \\
&=&\int_{B_{\delta}(0)}   \left(\frac{n-2k}{2}\phia^{2-\crit}\hua^2+\phia^{2-\crit}\frac{x^i}{2}\partial_i (\hua^2)\right)\, dx \\
&=&\int_{B_{\delta}(0)}  \left(\frac{n-2k}{2}\phia^{2-\crit} -\partial_i(\phia^{2-\crit}\frac{x^i}{2})\right) \hua^2\, dx+\int_{\partial B_{\delta}(0)}    \phia^{2-\crit}\frac{(x,\nu)}{2} \hua^2 \, d\sigma\\
%&=&\int_{B_{\delta}(0)}  \left(\frac{n-2k}{2}\phia^{2-\crit} -\partial_i(\phia^{2-\crit}\frac{x^i}{2})\right) \hua^2\, dx+O\left(\frac{\ma^{n-2k}}{\alpha^{2q}}\right)\\
%&=&\int_{B_{\delta}(0)}  \left(-k\phia^{2-\crit} - \frac{x^i\partial_i \phia^{2-\crit}}{2} \right) \hua^2\, dx+O\left(\frac{\ma^{n-2k}}{\alpha^{2q}}\right)\\
&=&-k\int_{B_{\delta}(0)}  \left(1+O(|x|)\right) \hua^2\, dx+O\left(\frac{\ma^{n-2k}}{\alpha^{2q}}\right)\end{eqnarray*}
And then, using the controls of $III_\alpha$ and $IV_\alpha$ and $III_\alpha-II_\alpha= IV_\alpha$, we get
\begin{equation}\label{eq:45}
c_{n,k}R_g(x_\alpha)\int_{B_\delta(0)} \left(   \Delta_\xi^{\frac{k-1}{2}} \hat{V}_\alpha\right)^2\,dx -k \alpha^{2k}\int_{B_{\delta}(0)}  \left(1+O(|x|)\right) \hua^2\, dx = O(\theta_\a)+O(\theta_\a^\prime).
\end{equation}
The explicit expression of $\hat{V}_\alpha(x)=\chi(\alpha|x|)U_{\ma}(x)$ for all $x\in \rn$ yields the existence of some $c_n^{(1)}>0$ such that
\begin{equation}\label{est:deltaudeux}
\int_{B_\delta(0)} \left(   \Delta_\xi^{\frac{k-1}{2}} \hat{V}_\alpha\right)^2\,dx =\left\{\begin{array}{cc}
\ma^2\left(\int_{\rn} \left(   \Delta_\xi^{\frac{k-1}{2}} U\right)^2\,dx +o(1)\right)&\hbox{ if }n>2k+2\\
(c_n^{(1)}+o(1))\ma^2\ln\frac{1}{\alpha\ma}  &\hbox{ if }n=2k+2\\
O\left(\frac{\ma}{\alpha}\right)&\hbox{ if }n=2k+1.
\end{array}\right.
\end{equation}
\subsection{Conclusion for $n\geq 2k+2$ and $R_g<0$}
As one checks,   $\theta_\alpha=o(\ma^2)$ when $n\geq 2k+2$; $\theta_\alpha'=o(\ma^2)$ when $n>2k+2$; and $\theta_\alpha'=O(\ma^2)$ when $n=2k+2$. Therefore, \eqref{eq:45} and \eqref{est:deltaudeux} yield
\begin{equation*}
(c_{n,k}R_g(x_\alpha)+o(1))\int_{B_\delta(0)} \left(   \Delta_\xi^{\frac{k-1}{2}} \hat{V}_\alpha\right)^2\,dx =k \alpha^{2k}\int_{B_{\delta}(0)}  \left(1+O(|x|)\right) \hua^2\, dx 
\end{equation*}
as $\alpha\to +\infty$. This implies that $R_g(x_\alpha)\geq 0$ for $\alpha\to +\infty$ and yields  a contradiction if $R_g$ is negative. Therefore \eqref{ineq:opt} holds when $n\geq 2k+2$ and $R_g< 0$ everywhere.

\smallskip\noindent For the other cases, we need to be more precise with the $L^2-$term in \eqref{eq:45}.

\subsection{Estimate of the $L^2-$norm of $\hua$}
Let us first assume that $n\geq 4k$. Since $k>1$, we take $\tau=2$ in \eqref{choice:tau}. Since $\hua=\hat{V}_\alpha+\hat{\phi}_\alpha$, We have that
$$\int_{B_\delta(0) }(1+O(|x|))\hua^2\, dx=\int_{B_\delta(0) } \hat{V}_\alpha^2\, dx+O\left(\int_{B_\delta(0) }|x|\hat{V}_\alpha^2\, dx+\int_{B_\delta(0) }\hat{\Theta}_\alpha U_{\ma}^2\, dx\right).$$
The second term of the right-hand-side is $o(\ma^{2k})$ if $n>4k$ and $O(\ma^{2k})$ if $n=4k$. The first term is computed explicitly since $\hat{V}_\alpha(x)=\chi(\alpha |x|)U_{\ma}(x)$ for all $x\in\rn$. Then there exists $c_n^{(2)}>0$ such that
\begin{equation*}
\alpha^{2k}\int_{B_\delta(0)} (1+O(|x|))\hua^2\, dx =\left\{\begin{array}{cc}
(\alpha\ma)^{2k}\left(\int_{\rn} U^2\,dx +o(1)\right)&\hbox{ if }n>4k\\
(c_n^{(2)}+o(1))(\alpha\ma)^{2k}\ln\frac{1}{\alpha\ma}  &\hbox{ if }n=4k\end{array}\right.
\end{equation*}
In order to deal with the case $n<4k$, we need an intermediate result:
\begin{lemma}\label{prop:rescal} We have that
\begin{equation*}
\lim_{\alpha\to +\infty} \frac{1}{\alpha^{n-2k}\ma^{\frac{n-2k}{2}}}\ua(\hbox{exp}^g_{\xa}(\alpha^{-1}\,\cdot \,))=c_U\Gamma \quad \hbox{in }C^{2k}_{loc}(\rn-\{0\}), 
\end{equation*}
where $c_U=\int_{\rn}U^{\crit-1}\, dx>0$ and $\Gamma$ is the Green's function for the operator $\D_\xi^k+1$ on $\rr^n$ given by Theorem \ref{th:green:rn}. 
\end{lemma}
{\it Proof of Lemma \ref{prop:rescal}.} For any $\alpha\geq 1$, we define 
$$\wa(X):=\frac{1}{\alpha^{n-2k}\ma^{\frac{n-2k}{2}}}\ua(\hbox{exp}_{\xa}(\alpha^{-1}X))\hbox{ for all }X\in B_{\alpha i_g}(0),$$
where in this proof, $\hbox{exp}_{\xa}=\hbox{exp}_{\xa}^g$ for simplicity. Equation \eqref{eq:ua} rewrites
\begin{equation*}
\Delta_{\bar{g}_\alpha}^k\wa+\wa=(\alpha\ma)^{2k}|\wa|^{\crit-2}\wa\hbox{ in }B_{\alpha i_g }(0),
\end{equation*}
where $\bar{g}_\alpha:=(\hbox{exp}_{\xa} )(\alpha^{-1}\cdot)$ converges $C^{2k}_{loc}$ to the Euclidean metric $\xi$. Moreover, the upper bound \eqref{eq:estimbndua}  for $l=0$ rewrites
\begin{equation*}
|\wa(X)|\leq C_q\frac{|X|^{2k-n}}{1+|X|^q}\hbox{ for all }X\in B_{\alpha i_g }(0).
\end{equation*}
Since $\alpha\ma\to 0$ as $\alpha\to +\infty$, it   follows from elliptic theory (see \cite{ADN} or the Appendix in \cite{robert:gjms}) that there is $  w_\infty\in C^\infty(\rn-\{0\})$ such that
\begin{equation}\label{cv:wa}
\lim_{\alpha\to +\infty}\wa=w_\infty\hbox{ in }C^{2k}_{loc}(\rn-\{0\})\hbox{ with }|w_\infty(X)|\leq C_q\frac{|X|^{2k-n}}{1+|X|^q}\hbox{ for all }X\in \rn.
\end{equation}
Let $\varphi\in C^\infty_c(\rn)$ be a test-function. We define $\varphi_\alpha:=\varphi(\alpha \hbox{exp}_{\xa}^{-1}(\cdot))$ for all $x\in M$. Integration by parts, a change of variable and equation \eqref{eq:ua} yield
\begin{equation}\label{eq:dist:wa}
\int_{\rn}\wa\left(\Delta_{\bar{g}_\alpha}^k\varphi+\varphi\right)\, dv_{\bar{g}_\alpha}=\ma^{-\frac{n-2k}{2}}\int_M|\ua|^{\crit-2}\ua \varphi_\alpha\, dv_g
\end{equation}
The control \eqref{eq:estimbndua}, \eqref{lim:tua:c2k:bis}, $U\in L^{\crit-1}(\rn)$ and $\tilde{g}_\alpha:=(\hbox{exp}_{\xa}^\star g)(\ma\cdot)\to \xi$ yield 
%
%Given $R>0$, the pointwise control \eqref{eq:estimbndua}  yields
%$$\left|\ma^{-\frac{n-2k}{2}}\int_{M-B_{R\ma }(\xa)}|\ua|^{\crit-2}\ua \varphi_\alpha\, dv_g\right|\leq C\ma^{2k}+C\int_{\rn-B_R(0)}(1+|x|^2)^{-\frac{n+2k}{2}}\, dx.$$
%On the ball, we have that
%$$ \ma^{-\frac{n-2k}{2}}\int_{B_{R\ma }(\xa)}|\ua|^{\crit-2}\ua \varphi_\alpha\, dv_g=\int_{B_{R  }(0)}|\va|^{\crit-2}\va \varphi (\alpha\ma X)\, dv_{\tilde{g}_\alpha}$$
%where  $\tilde{g}_\alpha:=(\hbox{exp}_{\xa}^\star g)(\ma\cdot)$ converges $C^{2k}_{loc}$ to the Euclidean metric $\xi$. Using \eqref{lim:tua:c2k:bis}, $U\in L^{\crit-1}(\rn)$, $\alpha\ma\to 0$ as $\alpha\to +\infty$, letting $\alpha\to +\infty$ in \eqref{eq:dist:wa} and then $R\to +\infty$ in the above expressions, we get that
\begin{equation*}%\label{eq:dist:winfinity}
\int_{\rn}w_\infty\left(\Delta_\xi^k\varphi+\varphi\right)\, dx=\left(\int_{\rn}U^{\crit-1}\, dx\right)\varphi(0)\hbox{ for all }\varphi\in C^\infty_c(\rn).
\end{equation*}
Since $w_\infty\in L^1(\rn)$ by \eqref{cv:wa}, it follows from the uniqueness in Theorem \ref{th:green:rn} that $w_\infty=\left(\int_{\rn}U^{\crit-1}\, dx\right)\Gamma$. This proves Lemma \ref{prop:rescal}.\qed

% and $\alpha>0$, and it is asymptotically proportional to $\ma^{2k}\ln\frac{1}{\alpha\ma}$. 
%
%
%For any $R>0$, it follows from \eqref{ineq:59} that
%$$\int_{B_\delta(0)-B_{R\ma}(0)}\hua^2\, dx\leq C\ma^2\int_{B_{\delta/\ma}-B_R(0)}\frac{dX}{(1+|X|^2)^{n-2k}}.$$
%Therefore, 
%$$\lim_{R\to +\infty}\lim_{\alpha\to +\infty}\ma^{-2}\int_{B_\delta(0)-B_{R\ma}(0)}\hua^2\, dx=0\hbox{ when }n>4k.$$
%It follows from \eqref{lim:tua:c2k:bis} that for all $R>0$, 
%$$\ma^{-2}\int_{ B_{R\ma}(0)}(1+O(|x|))\hua^2\, dx=\int_{ B_{R }(0)}(1+O(\ma|X|))v_\alpha^2\, dX\to \int_{B_R(0)}U^2\, dX\hbox{ as }\alpha\to +\infty.$$
%Since $U\in L^2(\rn)$ when $n>4k$, we get that
%$$\alpha^{2k}\int_{B_\delta(0) }(1+O(|x|))\hua^2\, dx=(\alpha\ma)^{2k}\left(\int_{\rn}U^2\, dX+o(1)\right)\hbox{ when }n>4k.$$

\medskip\noindent We come back to the case $n<4k$. We fix $R>1$. It follows from Lemma \ref{prop:rescal} that for any $\alpha>0$, we have that
\begin{eqnarray*}
\int_{B_{R\alpha^{-1}}(0)-B_{R^{-1}\alpha^{-1}}(0)}\hua^2\, dx&=&
%\alpha^{n-4k}\ma^{n-2k}\int_{B_{R }(0)-B_{R^{-1 }(0)}}\wa^2\, dx\\
%&=&
\alpha^{n-4k}\ma^{n-2k}\left(c_U^2\int_{B_{R }(0)-B_{R^{-1 }(0)}}\Gamma^2\, dx+o(1)\right).
\end{eqnarray*}
With the pointwise control \eqref{ineq:59}, we get that
\begin{eqnarray*}
&&\int_{(B_\delta(0)-B_{R\alpha^{-1}}(0))\cup B_{R^{-1}\alpha^{-1}}(0)}\hua^2\, dx
%&&\leq C\int_{ (B_\delta(0)-B_{R\alpha^{-1}}(0))\cup B_{R^{-1}\alpha^{-1}}(0)}\frac{\ma^{n-2k}}{(\ma+|x|)^{2(n-2k)}(1+\alpha|x|)^q)}\, dx\\
%&&\leq  C\int_{  B_{R^{-1}\alpha^{-1}}(0)}  \ma^{n-2k}|x|^{4k-2n}\, dx + C\int_{ B_\delta(0)-B_{R\alpha^{-1}}(0) }\frac{\ma^{n-2k}}{ |x| ^{2(n-2k)+q} \alpha^q}\, dx\\
\leq C \frac{\ma^{n-2k}\alpha^{n-4k}}{R^{4k-n}}+C \frac{\ma^{n-2k}}{\alpha^{4k-n}R^{-4k+n+q}}
\end{eqnarray*}
for $q$ large enough. Therefore, since $\Gamma\in L^2(\rn)$ for $n<4k$, letting $\alpha\to +\infty$ and then $R\to +\infty$, we get that
$$ \lim_{\alpha\to +\infty}\alpha^{2k}\int_{B_\delta(0) }(1+O(|x|))\hua^2\, dx=(\alpha\ma)^{n-2k}\left( c_U^2\int_{\rn} \Gamma^2\, dx+o(1)\right)\hbox{ if }n<4k.$$
%The limit case $n=4k$ requires to be a bit more precise and to use that $\hua=\hat{V}_\alpha+\hat{\phi}_\alpha$. Since $n=4k$ and $k>1$, we take $\tau=2$ in \eqref{choice:tau}. We have that
%$$\int_{B_\delta(0) }(1+O(|x|))\hua^2\, dx=\int_{B_\delta(0) } \hat{V}_\alpha^2\, dx+O\left(\int_{B_\delta(0) }|x|\hat{V}_\alpha^2\, dx+\int_{B_\delta(0) }\hat{\Theta}_\alpha U_{\ma}^2\, dx\right)$$
%and the second terme of the right-hand-side is $O(\ma^{2k})$ since $n=4k$. The first term is computed explicitly since $\hat{V}_\alpha(x)=\chi(\alpha |x|)U_{\ma}(x)$ for all $x\in\rn$ and $\alpha>0$, and it is asymptotically proportional to $\ma^{2k}\ln\frac{1}{\alpha\ma}$. Finally, we get that there exists  $c_n^{(2)}>0$ such that
Finally, all theses cases yield
\begin{equation}\label{est:udeux}
\alpha^{2k}\int_{B_\delta(0)} (1+O(|x|))\hua^2\, dx =\left\{\begin{array}{cc}
(\alpha\ma)^{2k}\left(\int_{\rn} U^2\,dx +o(1)\right)&\hbox{ if }n>4k\\
(c_n^{(2)}+o(1))(\alpha\ma)^{2k}\ln\frac{1}{\alpha\ma}  &\hbox{ if }n=4k\\
(\alpha\ma)^{n-2k}\left(c_U^2\int_{\rn} \Gamma^2\,dx +o(1)\right)&\hbox{ if }n<4k.
\end{array}\right.
\end{equation}
for some $c_n^{(2)}>0$.

\subsection{Conclusion and proof of the validity of \eqref{ineq:opt} in the remaining cases} We are now plugging \eqref{est:deltaudeux} and \eqref{est:udeux}  into the estimate \eqref{est:II} of $II_\alpha$ with the definition \eqref{def:theta} of $\theta_\alpha,\theta_\alpha^\prime$ and the choice \eqref{choice:tau} of $\tau$. Recall that   $k>1$ and then $\tau<2k$.

\smallskip\noindent{\bf Case 1: $n=2k+1$.} We choose   $0<\tau<1/2$. Then \eqref{eq:45} rewrites

$$O\left(\frac{\ma}{\alpha}\right) -k (\alpha\ma)\left(c_U^2\int_{\rn} \Gamma^2\,dx +o(1)\right) = O\left(\frac{\ma}{\alpha^3}\right)+O\left(\frac{\ma}{\alpha}\right),$$
so that $\alpha\ma=O\left(\frac{\ma}{\alpha}\right)$, which is a contradiction. Therefore \eqref{ineq:opt} holds when $n=2k+1$ regardless of the geometry.

\smallskip\noindent{\bf Case 2: $n\in \{2k+2,2k+3\}$.} The constraint \eqref{choice:tau} allows to take any $0<\tau<1$ for $n=2k+2$ and $1<\tau<3/2$ for $n=2k+3$. Here, \eqref{eq:45} rewrites
\begin{multline*}
    c_{n,k}R_g(x_\alpha) \int_{B_\delta(0)} \left(   \Delta_\xi^{\frac{k-1}{2}} \hat{V}_\alpha\right)^2\,dx -k (\alpha\ma)^{n-2k}\left(c_U^2\int_{\rn} \Gamma^2\,dx +o(1)\right)\\ = O\left(\alpha^{n-2k-2}\ma^{n-2k}\right),\end{multline*}
% and then 
%$$c_{n,k}R_g(x_\alpha) \int_{B_\delta(0)} \left(   \Delta_\xi^{\frac{k-1}{2}} \hat{V}_\alpha\right)^2\,dx -k (\alpha\ma)^2\left(c_U^2\int_{\rn} \Gamma^2\,dx +o(1)\right) = O\left(\frac{\ma^2}{\alpha^2}\right)+O\left(\ma^2\right),$$
and then
$$c_{n,k}R_g(x_\alpha) \int_{B_\delta(0)} \left(   \Delta_\xi^{\frac{k-1}{2}} \hat{V}_\alpha\right)^2\,dx =k (\alpha\ma)^{n-2k}\left(c_U^2\int_{\rn} \Gamma^2\,dx +o(1)\right).$$
This implies that $R_g(x_\alpha)> 0$ for $\alpha\to +\infty$ and yields a contradiction when $R_g\leq 0$. Therefore \eqref{ineq:opt} holds when $n\in\{2k+2,2k+3\}$ and $R_g\leq 0$ everywhere.

\smallskip\noindent{\bf Case 3: $n\geq 2k+4$ and $k=2$.} In this situation, the constraint \eqref{choice:tau} allows to take  $\tau=2$. Noting that $4k=2k+4=8$, \eqref{eq:45} rewrites

%Here, \eqref{est:II} rewrites
%
%\begin{eqnarray*}
%&&c_{n,k}R_g(x_\alpha)\ma^2\left(\int_{\rn} \left(   \Delta_\xi^{\frac{k-1}{2}} U\right)^2\,dX+o(1)\right) -k \alpha^{2k}\int_{B_{\delta}(0)}  \left(1+O(|x|)\right) \hua^2\, dx \\
%&&= O\left(\begin{array}{cc}
%\alpha^2\ma^4 &\hbox{ if }n> 2k+4,\\
%\alpha^2\ma^4\ln\frac{1}{\alpha\ma} &\hbox{ if }n = 2k+4.
%\end{array}\right).
%\end{eqnarray*}
%Since $\alpha\ma\to 0$ as $\alpha\to +\infty$, we get that
%$$(R_g(x_\alpha)+o(1))\ma^2=k \alpha^{2k}\int_{B_{\delta}(0)}  \left(1+O(|x|)\right) \hua^2\, dx,$$
%and then $R_g(x_\alpha)\geq o(1)$, so that $R_g(x_\infty)\geq 0$ where $x_\infty=\lim_{\alpha\to +\infty} x_\alpha$. We then get a contradiction if $R_g$ is negative. Therefore \eqref{ineq:opt} holds when $n\geq 2k+4$ and $R_g< 0$ everywhere.

%\smallskip\noindent When $k=2$, we can be a bit more precise. 

\begin{equation*}
c_{n,k}R_g(x_\alpha)\ma^2\left(\|\Delta_\xi^{\frac{k-1}{2}} U\|_{L^2}^2+o(1)\right) =\left\{\begin{array}{cc}
(\alpha\ma)^{4}\left(\int_{\rn} U^2\,dx +o(1)\right)&\hbox{ if }n>8\\
(c_n^{(2)}+o(1))(\alpha\ma)^{4}\ln\frac{1}{\alpha\ma}  &\hbox{ if }n=8
\end{array}\right.\end{equation*}
for some $c_n^{(2)}>0$. This is   again a contradiction when $R_g\leq 0$. Therefore \eqref{ineq:opt} holds when $k=2$, $n\geq 8$ and $R_g\leq 0$ everywhere.

\medskip\noindent The validity assertions of Theorems \ref{th:ok:Iopt:bis} and  \ref{th:ok:noIopt:bis} are consequences of these four cases except for the case of the flat torus.

\subsection{Validity of the optimal inequality on the flat torus}
Assume that \eqref{ineq:opt} is not valid on the flat torus $(\tn, g_{\tn})$, $n>2k$. With the same notations as above, there is no need to perform a conformal change and we take $g_p:=g_{\tn}$ with $\varphi_p\equiv 1$ for all $p\in\tn$. Via the exponential chart, the pull-back metric of $g_{\tn}$ is the Euclidean metric $\xi$. Therefore, \eqref{eq:hua} writes $\Delta_\xi^k\hua+\alpha^{2k} \hua=|\hua|^{\crit-2}\hua$  in $B_\delta(0)$. The   Pohozaev-Pucci-Serrin identity $III_\alpha-II_\alpha= IV_\alpha$ rewrites
$$\alpha^{2k}\int_{B_{\delta}(0)}  \hua  T(\hua)\, dx=II_\alpha=-IV_\alpha$$
since $III_\alpha=0$. As above, we have that $IV_\alpha=O\left(\frac{\ma^{n-2k}}{\alpha^{2q}}\right)$ and 
$$II_\alpha=\alpha^{2k}\left(-k\int_{B_{\delta}(0)}  \left(1+O(|x|)\right) \hua^2\, dx+O\left(\frac{\ma^{n-2k}}{\alpha^{2q}}\right)\right)$$
therefore, we get that $\int_{B_{\delta}(0)}  \left(1+O(|x|)\right) \hua^2\, dx=O\left(\frac{\ma^{n-2k}}{\alpha^{2q}}\right)$. This equality contradicts \eqref{est:udeux}.  So \eqref{ineq:opt} holds on the flat torus when $n> 2k$.

\medskip\noindent This proves the validity assertion of Theorem \ref{th:ok:noIopt:bis} on the flat torus.

% \section{Approximation near a bubble up to the Kernel}\label{sec:part1}
% %\subsection{Euclidean and Riemannian bubbles.}
% %Let $u \in \hSob(\rr^n)$ be a solution to the critical equation in $\rr^n$, 
% %\begin{equation}\label{eq:critRn}
% %    \Dg[\xi]^k u = |u|^{\crit-2} u,
% %\end{equation}
% %where $\xi$ is the Euclidean metric and $\Dg[\xi] := -\sum_{i=1}^n \partial_i^2$. If $u$ is positive, \cite{WeiXu99} showed that $u$ is equal, up to translations and dilations, to the positive Euclidean bubble
% %\[  \eBub(y) = \big(1 + \pk |y|^2\big)^{-\frac{n-2k}{2}},
% %    \]
% %where $\pk = \bpr{\Pi_{m=-k}^{k-1}(n+2m)}^{-\frac{1}{k}}$. The bubble is an extremal for the critical Sobolev inequality on $\rr^n$, and satisfies 
% %\[  \int_{\rr^n} \eBub^\crit\, dy = K(n,k)^{-\frac{n}{2k}}.
% %    \]
% %Moreover, when $u$ is a sign-changing solution to \eqref{eq:critRn}, it follows from \cite[Lemma 5]{GeWeiZhou11} that 
% %\begin{equation}\label{eq:Ldeusign}
% %    \int_{\rr^n} |u|^\crit\, dy \geq 2 K(n,k)^{-\frac{n}{2k}}.
% %    \end{equation}
% %%We now let $(M,g)$ be a compact Riemannian manifold without boundary, of dimension $n>2k$.
% Let $\TBub$ be defined as in \eqref{def:TBub}. Straightforward computations show that there exists $C>0$ such that 

\section{Pointwise estimates}\label{sec:part3}
Let $\TBub$ be defined as in \eqref{def:TBub}. In this section, we prove the existence and uniqueness of a perturbation term $\phi_\a$ such that $V_\a + \phi_\a$ is a solution of \eqref{eq:ua:th22} up to kernel elements. We obtain precise pointwise estimates on $\phi_\a$ and show that it is small compared to $\TBub$ in a precise sense, in Theorem \ref{prop:unicC0}, which is the main result of this section. Before stating the Theorem, we need some preliminary results and definitions.

Straightforward computations show that there exists $C>0$ such that 
\begin{multline}\label{eq:estRV}
    \abs{(\Dg^k + \a^{2k}) \TBub(x) - \TBub(x)^{\crit-1}} \leq C\Bub_{\zm}(x)\a^{2} (\mu + \dg{z,x})^{2-2k} {\bf 1}_{\a\dg{z,x} \leq 2}.
%    \\
%        \times \begin{cases}
%            \a^{2} (\mu + \dg{z,x})^{2-2k} & \text{when } \a\dg{z,x} \leq 2\\
%           % \a^{2k} & \text{when } 1 \leq \a\dg{z,x} \leq 2\\
%            0 & \text{when } \a\dg{z,x} > 2
%        \end{cases}.
\end{multline}
We also compute that
\begin{equation}\label{eq:RV0Hk}
    \Lnorm[(M)]{\frac{2n}{n+2k}}{(\Dg^k + \a^{2k}) \TBub - \TBub^{\crit-1}} \leq C \sigma_\alpha,
\end{equation}
for some $C>0$, where
$$\sigma_\alpha:=\left\{\begin{array}{cc}
(\alpha \ma)^2 &\hbox{ if }n>2k+4\\
(\alpha\ma)^2\ln\frac{1}{\alpha\ma}&\hbox{ if }n=2k+4\\
(\alpha \ma)^\frac{n-2k}{2} &\hbox{ if }n<2k+4
\end{array}\right. .$$
%
%\subsection{Kernel elements.}
We define $\Ker[] := \{h \in \hSob(\rr^n) \,:\, \Dg[\xi]^k h = (\crit-1) \eBub^{\crit-2} h\}$. By Bartsch-Weth-Willem \cite{BarWetWil03}, $\Ker[]$ is spanned by the $n+1$ linearly independent functions 
\begin{equation*}  
    \Ze{0}(y) := y\cdot \nabla\eBub(y) + \frac{n-2k}{2} \eBub(y)\hbox{ and }
    \Ze{j}(y) := \frac{\partial\eBub}{\partial y_j}(y) \qquad j = 1,\ldots, n
\end{equation*}
for $y \in \rr^n$. These correspond to the derivatives of $\eBub$ with respect to the parameters $\mu,z$ under which equation $\Delta_\xi^kV=V^{\crit-1}$ is invariant. It holds that for all $y\in \rr^n$,
\begin{equation}\label{eq:estimZe}  
    |\nabla_\xi^l \Ze{j}(y)| \leq C(1+|y|)^{2k-n-l} \qquad j = 0,\ldots , n, \quad l\geq 0.
\end{equation}
\begin{definition}
    Let $\a \geq \delta^{-1}$ and $\rho\leq 1$, $\pa_0 =(\zm[0])\in \param$, and let $\chi \in \Cct(\rr_{\geq 0})$ be such that $\chi \equiv 1$ in $[0,1)$ and $\chi \equiv 0$ in $(2,+\infty)$. We define 
    \begin{equation*}
        \Zeda[\pa_0]{0} := \mu_0 \bpr{\frac{\partial\TBub}{\partial\mu}}\rvert_{z=z_0,\, \mu=\mu_0}\hbox{ and }       \Zeda[\pa_0]{j} := \mu_0 \bpr{\frac{\partial\TBub}{\partial z_j}}\rvert_{z=z_0,\, \mu=\mu_0} 
    \end{equation*}
  for  $j=1,\ldots, n$.  By analogy with the Euclidean case, we also let 
    \begin{equation}\label{def:ker}
        \Kera[\pa_0] := \operatorname{span}\left\{\Zeda[\pa_0]{j}, \, 0\leq j \leq n\right\} \sub \Sob(M).
    \end{equation}
\end{definition}
In the spirit of  \eqref{def:TBub}, given $W\in D_k^2(\rn)$, we define
 \begin{equation*}%\label{def:transfo:W}  
        {\mathcal T}_{\alpha,\nu}(W)(x) := \chi(\a d_{g_z}(z,x))\, \varphi_z(x)\, \mu^{-\frac{n-2k}{2}} W\Big(\tfrac{1}{\mu}(\exp_z^{g_z})^{-1}(x)\Big)
        \end{equation*}
for all $x \in M$. For instance, $\TBub={\mathcal T}_{\alpha,\nu}(U)$. Standard computations yield $\Zeda[\pa_0]{j} =-{\mathcal T}_{\alpha,\nu}(Z^j)+o(1)$ for all $ j=0,...,n$
in $H_k^2(M)$. As in \eqref{eq:RV0Hk}, we obtain 
\begin{equation}\label{eq:RZ0Hk}
    \Lnorm[(M)]{\frac{2n}{n+2k}}{(\Dg^k + \a^{2k}) \Zeda{j} - (\crit-1)\TBub^{\crit-2}\Zeda{j}} \leq C \sigma_\alpha,
\end{equation}
for $j=0,\ldots, n$. We also have the following (see \cite{Car24}).
\begin{proposition}\label{prop:Zjort}
    Let $(\rho_i)_i, (\a_i)_i \in \rr_{>0}$ be such that $\a_i \geq \delta^{-1}$ and $\rho_i \to 0$ as $i\to \infty$, and $\pa_i = (\zm[i]) \in \param[i]$. Then for $j,j' \in \{0,\ldots ,n\}$, we get 
    \[  \int_M{(\Dg^k + \a_i^{2k}) \Zed[i]{j} \, \Zed[i]{j'}} \, dv_g\xrightarrow{i\to \infty} \delta_{j\,j'}\int_{\rn}(\Delta_\xi^{\frac{k}{2}}\Ze{j})^2\, dx=\delta_{j\,j'}\Hnorm{k}{\Ze{j}}^2.
        \]
\end{proposition} 

%\subsection{Uniqueness of solutions near a bubble}\label{sec:part2}
\medskip\noindent We first state existence and uniqueness of solution to \eqref{eq:ua:th22} 
up to terms in $\Ker$, of the shape $u = \TBub + \phi_{\zma}$, for $\phi_{\zma} \in \Ker^\perp$ small. In the sequel, $\proKe$ denotes the projection onto $\Ker^\perp$ in $\Sob(M)$ with respect to the   scalar product  $\vprod[\Sob(M)]{u}{v} := \sum_{l=0}^k\intM{\vprod[]{\Dg^{l/2} u}{\Dg^{l/2}v}}$. For $\a \geq \delta^{-1}$, $\rho\leq 1$, and $\pa \in \param$, we define the  operator $\Lia : \Ker^\perp \to \Ker^\perp$ as
\[  \Lia \phi := \proKe\bsq{\phi - (\Dg^k + \a^{2k})^{-1} \bpr{(\crit-1)\TBub^{\crit-2}\phi}}
    \]
Following the same arguments as in \cites{Pre22, Pre24,Car24}, we get the two following propositions:
\begin{proposition}\label{prop:invertLHk}
    There exists $\rho_0>0$ and $C_0>0$ such that, for all $\a \geq \delta^{-1}$, $\rho \leq \rho_0$, and $\pa \in \param$, $\Lia$ is invertible and 
    \[  \frac{1}{C_0} \Snorm{k}{\phi} \leq \Snorm{k}{\Lia \phi} \leq C_0 \Snorm{k}{\phi} \qquad \forall ~\phi \in \Ker^\perp. 
        \]
\end{proposition}
%The proof is standard and follows the same arguments as in \cites{Pre22, Pre24,Car24}.
\begin{proposition}\label{prop:unicHk}
    There exists $R_0 >0$, $\rho_0>0$ such that for all $\a \geq \delta^{-1}$, $\rho\leq \rho_0$, and $\pa \in \param$, there exists a unique solution $ \phi_{\zma} \in \Ker^\perp \cap \{\phi \in \Sob(M), \, \Snorm{k}{\phi} \leq R_0\}$    to 
    \begin{equation*}%\label{eq:modcritMHk}
        \proKe\bpr{\TBub + \phi_{\zma} - (\Dg^k + \a^{2k})^{-1} \bpr{|\TBub + \phi_{\zma}|^{\crit-2}(\TBub+\phi_{\zma})}} = 0.
    \end{equation*}
\end{proposition} 
%The proof follows from a fixed-point argument, as found in \cites{Pre22, Pre24,Car24}.

We now obtain precise pointwise estimates on the remainder term $\phi_{\zma}$ given by Proposition \ref{prop:unicHk}. 
Let $p\geq 1$, $\a \geq \delta^{-1}$, $\tau>0$ such that    \eqref{choice:tau} holds.
%, that is
%\begin{equation*}
%0<\tau\leq \min \left\{2,\frac{n-2k}{2}\right\}.
%\end{equation*}
Recall that $\tau<2k$ since $k>1$. For $\pa \in \param$, we define 
\begin{equation*}%\label{def:thet}
    \Tet(x) := \frac{\bpr{\a (\mu + \dg{z,x}) }^{\tau}}{\bpr{ 1+ \a \dg{z,x}}^{p}}.
\end{equation*}
Moreover, we define the following weighted norms, depending on $p,\tau,\alpha,\pa$:
\begin{equation}\label{def:inorms}\begin{aligned} 
    \inorm{u} &:= \Lnorm[(M)]{\infty}{\frac{u}{\Tet \Bub_{\zm}}}\\
    \iinorm{u} &:= \Lnorm[(M)]{\infty}{\frac{u}{\Tet~ (\mu+\dg{z,\cdot})^{-2k}\Bub_{\zm}}}\\
\end{aligned} \qquad \forall~ u \in C^0(M).\end{equation}
\begin{theorem}\label{prop:unicC0}
    Let $(\a_i)_i,\,(\rho_i)_i \in \rr_{>0}$ be such that $\a_i \to \infty$ and $\rho_i \to 0$ as $i\to \infty$, and let $\pa_i = (\zm[i]) \in \param[i]$. Let $p \geq 1$, there exists $\Lambda >0$ and $i_0$ such that for all $i\geq i_0$, there is a unique solution
    \[  \phi_i \in \Keri^\perp \cap \big\{\phi \in C^0(M),\, \inorm{\phi_i} \leq \Lambda\big\},
        \]
    to
    \begin{equation}\label{eq:modcritC0}  
        \proKi\bpr{\TBub[i] + \phi_i - (\Dg^k + \a_i^{2k})^{-1} \bpr{|\TBub[i] + \phi_i|^{\crit-2} (\TBub[i]+ \phi_i)}} = 0,
    \end{equation}
    where $\Keri$ is as defined in \eqref{def:ker}. Moreover, $\phi_i \in C^{2k-1}(M)$ and for any $p \geq 1$ and $\tau$ as in \eqref{choice:tau} with $k>1$, there exists $C>0$ such that for all $l=0,\ldots , 2k-1$, we have 
    \begin{equation}\label{bnd:phi}  (\mu_i + \dg{z_i,x})^l |\nabla_g^l \phi_i(x)| \leq C\, \Tet[i](x)\Bub_{\zm[i]}(x).
        \end{equation}
In particular,  $\Snorm{k}{\phi_i} \to 0$ as $  i \to \infty$.\end{theorem}
To prove Theorem \ref{prop:unicC0}, we first need to study the linearized version of the equation. 

\subsection{Linear blow-up analysis.} 
The purpose of this section is to prove a uniform invertibility result analog to Proposition \ref{prop:invertLHk}, but now in weighted pointwise spaces. Let $\rho\leq \rho_0$ be given by Proposition \ref{prop:invertLHk}, and $\a\geq \delta^{-1}$, $\pa\in \param$. Then for all $R \in C^0(M)$, there exists a unique function $\phi \in \Ker^\perp$ and a unique family $\{\lambda^j \in \rr, \, j=0,\ldots, n\}$, such that 
\begin{equation}\label{eq:modlin}
    (\Dg^k + \a^{2k})\phi - (\crit-1)\TBub^{\crit-2}\phi = R + \sum_{j=0}^n \lambda^j (\Dg^k+ \a^{2k})\Zed{j}.
\end{equation}
Moreover, the $\Sob(M)$ norm of $\phi$ is controlled by $\Vert R\Vert_{L^{\frac{2n}{n+2k}}(M)}$. This can be improved : When $R$ has suitable pointwise bounds, $\phi$ inherits better pointwise estimates too. 
\begin{proposition}\label{prop:estlinC0}
    Let $p\geq 1$, there exists $\rho_0>0$, $\a_0\geq \delta^{-1}$ and $C_0>0$ such that, for all $\a\geq \a_0$, $\rho \leq \rho_0$, and $\pa \in \param$, the following holds. Let $R \in C^0(M)$, and $\phi \in \Ker^\perp$ be the unique solution to \eqref{eq:modlin} given by Proposition \ref{prop:invertLHk}. Then $\phi \in C^0(M)$ and $ \inorm{\phi} \leq C_0\, \iinorm{R}$, where $\inorm{\cdot}$ and $\iinorm{\cdot}$ are as defined in \eqref{def:inorms}.
\end{proposition}
The proof adapts the arguments of Carletti \cite{Car24}, and follows ideas introduced in Premoselli \cites{pre:ccm, Pre22, Pre24}.
\begin{proof}
    We argue by contradiction. Assume that there are sequences $(\a_i)_i$, $(\rho_i)_i$ such that $\a_i \to \infty$ and $\rho_i \to 0$, $\pa_i = (\zm[i]) \in \param[i]$. Assume also that there exists a sequence $(R_i)_i$ of continuous functions in $M$ such that the following holds:   for all $i$ big enough, there exists a family $\phi_i \in \Keri^\perp$ which is the unique solution to \eqref{eq:modlin} given by Proposition \ref{prop:invertLHk}, for some unique $(\lambda_i^j)_i$, $j=0,\ldots ,n$, and  
    \begin{equation}\label{eq:contraphiRi}  
       \iinorm{R_i} = 1\hbox{ for all }i\, ;\,    \inorm{\phi_i} \to \infty  \text{ as }i\to \infty.
    \end{equation}
    Remark that by a direct computation, using $\tau$ as in \eqref{choice:tau}, \eqref{eq:contraphiRi} gives that 
    \begin{equation}\label{ineq:44}
\Lnorm[(M)]{\frac{2n}{n+2k}}{R_i} \leq C(\a_i\mu_i)^{\tau}\left\{\begin{array}{cc}
1&\hbox{ if }\tau<\frac{n-2k}{2}\\
\left(\ln\frac{1}{\alpha_i\mu_i}\right)^{\frac{n+2k}{2n}}&\hbox{ if }\tau=\frac{n-2k}{2}.\end{array}\right.
\end{equation}
By Proposition \ref{prop:invertLHk} and the dual embedding $L^{\frac{2n}{n+2k}}(M) \emb \Sob[-k](M)$, we get
    \begin{equation}\label{ineq:45}
     \Snorm{k}{\phi_i} \leq C \Snorm{-k}{R_i} \leq C\Lnorm[(M)]{\frac{2n}{n+2k}}{R_i} .
       \end{equation}
    The proof now splits in several steps.
    
    \smallskip\noindent\textbf{Step 0 : $\sum_{j=0}^n |\lambda_i^j| \to 0$.} Let $j_0 \in \{0,\ldots, n\}$, we start by testing equation \eqref{eq:modlin} against $\Zed[i]{j_0}$. Integrating by parts, using \eqref{eq:RZ0Hk}, \eqref{ineq:44}, and Proposition \ref{prop:Zjort}, we obtain
\begin{eqnarray*}
\lambda_i^{j_0} + o\Big(\sum_{j=0}^n |\lambda_i^{j}|\Big)& =& O(\Vert \phi_i\Vert_{\frac{2n}{n-2k}}\sigma_\alpha)+O\left(\int_M |R_i| \cdot |\Zed[i]{j_0}|\, dv_g\right).\\
& =& O((\a_i\mu_i)^{\tau})+O\left(\int_M |R_i| \cdot |\Zed[i]{j_0}|\, dv_g\right).
\end{eqnarray*}
where we have used \eqref{ineq:44}, \eqref{ineq:45}, Sobolev's inequality \eqref{ineq:AB} and $\tau\leq\frac{n-2k}{2}$. Regarding the second term of the right-hand-side, using the pointwise control \eqref{eq:estimZe}, the definition of $\Vert\cdot\Vert_{**}$, and $ \Vert R_i\Vert_{**}<C$, we have that
\begin{equation*}
\int_M |R_i \Zed[i]{j_0}|\, dv_g
%&\leq& C \int_M \Vert R_i\Vert_{**}\frac{ \Tet[i](x)\Bub_{\zm[i]}(x)}{(\mu_i+d_g(x, z_i)^{2k}}\frac{\mu_i^{\frac{n-2k}{2}}}{(\mu_i+d_g(x, z_i)^{n-2k}}\, dv_g\\
\leq C \int_M \frac{(\alpha_i (\mu_i+d_g(x,z_i)))^\tau}{(1+\alpha_i d_g(x,z_i))^p}\frac{\mu_i^{n-2k}}{(\mu_i+d_g(x, z_i))^{2n-2k}}\, dv_g\leq C(\a_i\mu_i)^{\tau}
\end{equation*}
since $\tau<n-2k$. Therefore   $\lambda_i^{j_0} + o\Big(\sum_{j=0}^n |\lambda_i^{j}|\Big) = O((\a_i\mu_i)^{\tau})$.  Thus, summing over all $j_0$, we have
    \begin{equation}\label{eq:estlambi}
        \sum_{j=0}^n |\lambda_i^j| \leq C (\a_i\mu_i)^{\tau}.
    \end{equation}
   \noindent\textbf{Step 1 : A first estimate on $|\phi_i|$.} Since $\phi_i$ solves \eqref{eq:modlin}, we have that
    \[  (\Dg^k + \a_i^{2k}) \bsq{\phi_i - \sum_{j=0}^n \lambda_i^j \Zed[i]{j}} = (\crit-1)\TBub[i]^{\crit-2} \phi_i + R_i.
        \]
Elliptic theory yields $\phi_i \in C^0(M)$. Given $x \in M$, Green's  formula yields 
    \begin{multline} \label{eq:reprphii}
        \phi_i(x) - \sum_{j=0}^{n}\lambda_i^j \Zed[i]{j}(x) = (\crit-1)\intM{G_{\alpha_i}(x,y)\TBub[i]^{\crit-2}(y)\phi_i(y)}(y)\\ + \intM{G_{\alpha_i}(x,y) R_i(y)}(y),
        \end{multline} 
    where $G_{\alpha_i}$ is the Green's function for the operator $\Dg^k + \a_i^{2k}$ as defined in Theorem \ref{th:green:M}. We define $\Psp(t) := (1+ t)^{-p}$ for all $t\geq 0$. Lemma \ref{lem:giraud} yields %and Lemmae A.1 and A.2 in Carletti \cite{carletti:green}
    \begin{multline*}  
        \intM{|G_{\alpha_i}(x,y)||R_i(y)|}(y)\\
        \begin{aligned}
            &\leq \iinorm{R_i} \intM{|G_{\alpha_i}(x,y)|\Tet[i](y)(\mu_i+\dg{z_i,y})^{-2k}\Bub_{\zm[i]}(y)}(y)\\
            &\leq C\Tet[i](x)\,\Bub_{\zm[i]}(x),
        \end{aligned}
    \end{multline*}
    and similarly
    \begin{equation*}  
        \intM{|G_{\alpha_i}(x,y)|\TBub[i]^{\crit-2}(y)|\phi_i(y)|}(y) \leq C \inorm{\phi_i}(\a_i\mu_i )^{\tau}\Psp(\a_i\dg{z_i,x}) \, \Bub_{\zm[i]}(x),
    \end{equation*}
    for all $x\in M$. Thus, \eqref{eq:reprphii} becomes, together with \eqref{eq:estlambi},
    \begin{equation}\label{eq:1stestphi}
        |\phi_i(x)| \leq C\Tet[i](x)\, \Bub_{\zm[i]}(x) + C \inorm{\phi_i}(\a_i\mu_i )^{\tau}\Psp(\a_i\dg{z_i,x})\, \Bub_{\zm[i]}(x).
    \end{equation}
We  set $\Tg_i := \exp_{z_i}^* g(\mu_i \cdot\,)$ and for all $y \in \Bal{0}{\delta/\mu_i} \sub \rr^n$, we define
    \[  \psi_i(y) := \frac{\mu_i^{\frac{n-2k}{2}}}{(\a_i\mu_i)^{\tau} \inorm{\phi_i}} \phi_i(\exp_{z_i}(\mu_i y)).
        \]
\smallskip\noindent\textbf{Step 2 : We claim that $\psi_i \to \psi_\infty$ in $C^0_{loc}(\rr^n)$.} Fox $R>0$, \eqref{eq:1stestphi}  gives 
    \begin{equation}\label{eq:boundpsi}
        |\psi_i(y)| \leq C (1+|y|)^{2k-n} + o(1)\qquad\forall~ y \in \Bal{0}{R},
    \end{equation}
    as $i\to \infty$, for some constant $C>0$ independent of $i$ and $R>0$. We get
    \[  \abs{(\Dg[\Tg_i]^k + \a_i^{2k} \mu_i^{2k})\psi_i(y)} = \frac{\mu_i^{\frac{n+2k}{2}}}{(\a_i\mu_i )^{\tau} \inorm{\phi_i}} \abs{(\Dg^k + \a_i^{2k})\phi_i(\exp_{z_i}(\mu_i y))}.
        \]
    Now, straightforward computations show that for all $y \in \Bal{0}{R}$,
    \begin{align}
        \notag &\mu_i^{2k} \TBub[i]^{\crit-2}(\exp_{z_i}(\mu_i y)) \leq 1\\
        \label{tmp:inftynormR} &\frac{\mu_i^{\frac{n+2k}{2}}}{(\a_i\mu_i )^{\tau} \inorm{\phi_i}} |R_i(\exp_{z_i}(\mu_i y))| \leq C\frac{\iinorm{R_i}}{\inorm{\phi_i}} = o(1)\\
        \notag &\frac{\mu_i^{\frac{n+2k}{2}}}{(\a_i\mu_i )^{\tau} \inorm{\phi_i}}|\lambda_i^j| \abs{(\Dg^k + \a_i^{2k})\Zed[i]{j}(\exp_{z_i}(\mu_i y))} \leq C \frac{1}{\inorm{\phi_i}} = o(1).
    \end{align}
    Using \eqref{eq:modlin}, it then follows that $\abs{(\Dg[\Tg_i]^k + \a_i^{2k} \mu_i^{2k})\psi_i(y)} \leq C $ for all $y \in \Bal{0}{R}$. Since $(\Dg[\Tg_i]^k+\a_i^{2k}\mu_i^{2k})$ is an elliptic operator with bounded coefficients, uniformly in $i$, then there exists $\psi_\infty \in C^{2k-1}(\rr^n)$ such that $\psi_i \to \psi_\infty$ in $C^{2k-1}_{loc}(\rr^n)$ up to a subsequence. Passing to the limit in \eqref{eq:boundpsi} yields
    \begin{equation}\label{eq:bndpsiinfty}  
        |\psi_\infty(y)| \leq C(1+ |y|)^{2k-n} \qquad \forall y \in \rr^n.
    \end{equation}

    \noindent\textbf{Step 3: $\psi_\infty\in \Ker[]-\{0\}$.}
    Integrating against a test-function in $\Cct(\rr^n)$, we get that, up to a subsequence 
    \begin{align*}
        \frac{\mu_i^{\frac{n+2k}{2}}}{(\a_i\mu_i )^{\tau} \inorm{\phi_i}}& (\Dg^k + \a_i^{2k})\phi_i(\exp_{z_i}(\mu_i \,\cdot\,)) \to \D_\xi^k \psi_\infty\\
        \frac{\mu_i^{\frac{n+2k}{2}}}{(\a_i\mu_i )^{\tau} \inorm{\phi_i}}& \TBub[i]^{\crit-2}(\exp_{z_i}(\mu_i \,\cdot\,))\phi_i(\exp_{z_i}(\mu_i \,\cdot\,)) \to \eBub^{\crit-2} \psi_\infty
    \end{align*}
    in the weak sense. Since $\phi_i$ solves \eqref{eq:modlin}, and using \eqref{tmp:inftynormR}, we obtain that $\psi_\infty$ solves the limit equation $\Dg[\xi]^k \psi_\infty = (\crit-1)\eBub^{\crit-2} \psi_\infty$,  in the distributional sense on $\rr^n$. Now since $\psi_\infty \in L^\infty(\rr^n)$ by \eqref{eq:bndpsiinfty}, a standard bootstrap argument gives that $\psi_\infty \in C^\infty(\rr^n)$. Then, using a representation formula, we also obtain $\psi_\infty \in \hSob(\rr^n)$, so that $\psi_\infty \in \Ker[]$.
    
    \smallskip\noindent We now claim that $\psi_\infty \not\equiv 0$.
    Indeed, let $y_i \in M$ be the point where $\frac{|\phi_i|}{\Tet[i]\Bub_{\zm[i]}}$ reaches its maximum. By \eqref{eq:1stestphi}, we get that
    \[  \inorm{\phi_i} = \frac{|\phi_i(y_i)|}{\Tet[i](y_i)\Bub_{\zm[i]}(y_i)} \leq C \bpr{1+ (\a_i\mu_i )^{\tau} \frac{\Psp(\a_i \dg{z_i,y_i})}{\Tet[i](y_i)}\inorm{\phi_i}},
        \] 
    this implies that $\frac{\dg{z_i,y_i}}{\mu_i} = \bigO(1)$ since $\inorm{\phi_i}\to\infty$ as $i\to\infty$. We can then assume that there exists $y_\infty \in \rr^n$ such that $\frac{1}{\mu_i}\exp_{z_i}^{-1}(y_i) \to y_\infty$ as $i\to \infty$, up to a subsequence. Then, since by definition of $\psi_i$, we have $\left|\psi_i\big(\frac{1}{\mu_i}\exp_{z_i}^{-1}(y_i)\big)\right| = c+o(1)$ for some $c>0$ for all $i$, it holds that $|\psi_\infty(y_\infty)| = c>0$, which yields $\psi_\infty \not\equiv 0$ and the claim is proved. 

    \smallskip\noindent\textbf{Step 4: Getting a contradiction.} Using that $\phi_i \in \Keri^\perp$, we have 
    \begin{equation}\label{tmp:phiinkerperp}  \frac{1}{(\a_i\mu_i )^{\tau}\inorm{\phi_i}} \sum_{l=0}^k \intM{\phi_i~ \Dg^l \Zed[i]{j}} = 0
        \end{equation}
    for $j=0,\ldots, n$.
    Now, by definition of $\Zed[i]{j}$ and using \eqref{eq:1stestphi}, we get that 
  $$\lim_{i\to +\infty}\frac{1}{(\a_i\mu_i )^{\tau}\inorm{\phi_i}} \sum_{l=0}^k \intM{\phi_i~ \Dg^l \Zed[i]{j}}=\int_{\rr^n} \psi_\infty ~\Dg[\xi]^k \Ze{j}\,dy .$$
%    
%    for $l=0,\ldots , k-1$,
%    \begin{eqnarray*}  
%        \frac{1}{(\a_i\mu_i)^{\tau}\inorm{\phi_i}} \abs{\intM{\phi_i~ \Dg^l \Zed[i]{j}}}
%        &\leq& C \int_{B_{2/\a_i}(z_i)} \frac{\Tet[i]}{(\a_i\mu_i )^{\tau}} \Bub_{\zm[i]}|\Dg^l \Zed[i]{j}|\, dv_g\\
%            &\leq & o(1)
%        \end{eqnarray*}
%    as $i\to \infty$. In a second step, for $l=k$, fix $R>0$, we have that
%    \[  \frac{1}{(\a_i\mu_i )^{\tau}\inorm{\phi_i}} \intM[\Bal{z_i}{R\mu_i}]{\phi_i~ \Dg^k \Zed[i]{j}} = \int_{\Bal{0}{R}} \psi_\infty~\Dg[\xi]^k \Ze{j} \, dy + o(1),
%        \]
%    while we also compute, as before
%    \begin{multline*}  
%        \frac{1}{(\a_i\mu_i )^{\tau}\inorm{\phi_i}} \abs{\intM[\Bal{z_i}{2/\a_i}\setminus \Bal{z_i}{R\mu_i}]{\phi_i~ \Dg^k \Zed[i]{j}}}\\ \leq C \int_{\rr^n \setminus \Bal{0}{R}} (1+|y|)^{-2n+2k+\tau} dy = \epsilon(R),
%        \end{multline*}
%    where $\epsilon(R) \to 0$ as $R\to \infty$, since $n-2k>\tau$. Using \eqref{eq:estimZe} and \eqref{eq:bndpsiinfty}, we get by dominated convergence that
%    \[  \lim_{R\to \infty} \int_{\Bal{0}{R}} \psi_\infty ~\Dg[\xi]^k \Ze{j}\,dy = \int_{\rr^n} \psi_\infty~\Dg[\xi]^k \Ze{j}\, dy. 
%        \]
%    Finally, letting first $i\to \infty$ and then $R\to \infty$ in the above, \eqref{tmp:phiinkerperp} gives 
Therefore $\int_{\rr^n} \psi_\infty ~\Dg[\xi]^k \Ze{j}\,dy = 0$ for all $j=0,\ldots, n$. Since $\psi_\infty \in \hSob(\rr^n)$, this shows that $\psi_\infty \in \Ker[]^\perp$, and this is a contradiction with the fact that $\psi_\infty \in \Ker[]$ and $\psi_\infty \not\equiv 0$. This proves Proposition \ref{prop:estlinC0}.
\end{proof}

\subsection{Proof of Theorem \ref{prop:unicC0}.}
The proof closely follows the arguments in \cite{Car24}.  Writing $f(t) := |t|^{\crit-2}t$ for $t\in \rr$, there exists $C>0$ and $0< \theta< \min(1,\crit-2)$ such that for all $b>0$, $a_1,a_2 \in \rr$, the following holds
\begin{equation}\label{eq:estpoint1} 
    |f(b+a_1) - f(b) - f'(b)a_1| \leq C |a_1|\bpr{|a_1|^{\crit-2} + b^{\crit-2-\theta}|a_1|^{\theta}},
\end{equation}
and 
\begin{multline}\label{eq:estpoint2}
    |f(b+a_1) - f(b+a_2) - f'(b)(a_1 - a_2)|\\
        \leq C |a_1 - a_2| \bpr{b^{\crit-2-\theta}|a_1|^\theta + |a_1|^{\crit-2} + b^{\crit-2-\theta}|a_2|^\theta + |a_2|^{\crit-2}}.
\end{multline}
%\begin{proof}[Proof of Proposition \ref{prop:unicC0}] 
We go back to the proof of the Theorem. Since $\tau\leq 2$, the bound \eqref{eq:estRV} yields the existence of $C_1>0$ such that
    \begin{equation}\label{tmp:estRV}
        \iinorm{(\Dg^k + \a_i^{2k}) \TBub[i] - \TBub[i]^{\crit-1}} \leq C_1.
        \end{equation}
    We define the set 
    \begin{multline*}  
        \mathcal{S}_i := \Bigg\{ \phi\in C^0(M) \,:\, \inorm{\phi} \leq 2 C_0C_1\\
            \text{and} \qquad \sum_{l=0}^k \intM{\phi\,\Dg^l \Zed[i]{j}} = 0 \quad \text{for } j=0,\ldots, n\Bigg\},
        \end{multline*}
    where $C_0>0$ is given by Proposition \ref{prop:estlinC0}, for all $i\geq i_0$. As one checks, $(\mathcal{S}_i, \inorm{\cdot})$ is a complete metric space for all $i\geq i_0$. For any $\phi \in \mathcal{S}_i$, we define %as in ... 
    \begin{equation*}%\label{def:Giphi}  
    N_{\alpha_i,\nu_i}(\phi) := |\TBub[i]+ \phi|^{\crit-2} (\TBub[i]+ \phi) - \TBub[i]^{\crit-1} - (\crit-1)\TBub[i]^{\crit-2} \phi,
    \end{equation*}
    and we set 
    \[  R_i(\phi) := -(\Dg^k + \a_i^{2k})\TBub[i] + \TBub[i]^{\crit-1} + N_{\alpha_i,\nu_i}(\phi).
        \]
    It holds that $R_i(\phi) \in C^0(M)$. Let $\Phi_i(\phi) \in \Keri^\perp$ be the unique solution in $\Keri^\perp$, given by Proposition \ref{prop:unicHk}, to 
    \[  (\Dg^k + \a_i^{2k}) \Phi_i(\phi) - (\crit-1)\TBub[i]^{\crit-2} \Phi_i(\phi) = R_i(\phi) + \sum_{j=0}^n \lambda_i^j (\Dg^k + \a_i^{2k}) \Zed[i]{j},
        \]
    for some  $\lambda_i^j \in \rr$, $j = 0,\ldots, n$. Standard elliptic theory (see \cite{ADN}) gives that $\Phi_i(\phi) \in C^{2k-1}(M)$. We show that $\Phi_i$ is contracting on $\mathcal{S}_i$.
    
    \smallskip\noindent\textbf{Step 1 : $\Phi_i$ stabilizes $\mathcal{S}_i$.} Using \eqref{eq:estpoint1}, we have
    \[  |N_{\alpha_i,\nu_i}(\phi)(x)| \leq C |\phi(x)| \bpr{|\phi(x)|^{\crit-2} + \TBub[i](x)^{\crit-2-\theta} |\phi(x)|^\theta}
        \]
    for some $0<\theta<\min(1,\crit-2)$. Therefore, straightforward computations show that for $\phi_i \in \mathcal{S}_i$, and $x\in M$, then
    $$  |N_{\alpha_i,\nu_i}(\phi)(x)| \leq C (\a_i \mu_i )^{\tau\theta} \a_i^{\tau} (\mu_i + \dg{z_i,x})^{\tau-2k} \Bub_{\zm[i]}(x)\hbox{ if }\a_i\dg{z_i,x} \leq 1,$$
    $$|N_{\alpha_i,\nu_i}(\phi)(x)| \leq C \a_i^{2k} \mu_i^{2k} \a_i^{\tau-p} \dg{z_i,x}^{\tau-2k-p} \Bub_{\zm[i]}(x)\hbox{ if }\a_i\dg{z_i,x} \geq 1.$$
%    the following holds :
%    \begin{itemize}
%        \item When $x\in M$ is such that $\a_i\dg{z_i,x} \leq 1$,
%        \[  |N_{\alpha_i,\nu_i}(\phi)(x)| \leq C (\a_i \mu_i )^{\tau\theta} \a_i^{\tau} (\mu_i + \dg{z_i,x})^{\tau-2k} \Bub_{\zm[i]}(x);
%            \] 
%        \item When $x\in M$ is such that $\a_i\dg{z_i,x} \geq 1$,
%        \[  |N_{\alpha_i,\nu_i}(\phi)(x)| \leq C \a_i^{2k} \mu_i^{2k} \a_i^{\tau-p} \dg{z_i,x}^{\tau-2k-p} \Bub_{\zm[i]}(x).
%            \]
%    \end{itemize}
    Since $\a_i \mu_i  \leq \rho_i \to 0$, there exists a sequence $(\epsilon_i)_i$ such that $\epsilon_i \to 0$, and with \eqref{tmp:estRV}, $R_i(\phi)$ satisfies $ \iinorm{R_i(\phi)} \leq (C_1 + \epsilon_i)$.  Proposition \ref{prop:estlinC0} yields $\inorm{\Phi_i(\phi)} \leq C_0(C_1 + \epsilon_i)$. Up to increasing $i_0$ to have $\epsilon_i < C_1$ for all $i\geq i_0$, since $\Phi_i(\phi) \in \Keri^\perp \cap C^{2k-1}(M)$ we get that $\Phi_i(\phi) \in \mathcal{S}_i$ for all $\phi \in \mathcal{S}_i$.
    
\smallskip\noindent\textbf{Step 2: $\Phi_i$ is contracting on $\mathcal{S}_i$.} Let $\phi_1,\phi_2 \in \mathcal{S}_i$, we have by definition of $\Phi_i$ that 
    \begin{multline*}  
        (\Dg^k + \a_i^{2k}) \Big(\Phi_i(\phi_1) - \Phi_i(\phi_2)\Big) - (\crit-1)\Big(\Phi_i(\phi_1)-\Phi_i(\phi_2)\Big)\\
            = N_{\alpha_i,\nu_i}(\phi_1) - N_{\alpha_i,\nu_i}(\phi_2) + \sum_{j=0}^n \Tilde{\lambda}_i^j (\Dg^k + \a_i^{2k})\Zed[i]{j}
        \end{multline*}
for some $\Tilde{\lambda}_i^j \in \rr$, $j=0,\ldots, n$. With \eqref{eq:estpoint2}, we get 
\[  \iinorm{N_{\alpha_i,\nu_i}(\phi_1)-N_{\alpha_i,\nu_i}(\phi_2)} \leq \epsilon_i \inorm{\phi_1-\phi_2}.
\] 
Thus, Proposition \ref{prop:estlinC0} yields $\inorm{\Phi_i(\phi_1)-\Phi_i(\phi_2)} \leq C_0 \epsilon_i \inorm{\phi_1-\phi_2}$. Up to increasing again $i_0$ to have $C_0\epsilon_i <1$ for all $i\geq i_0$, we conclude that $\Phi_i$ is a contraction on $\mathcal{S}_i$.

\smallskip\noindent\textbf{Step 3 : End of the proof.} Banach's fixed-point theorem applies and gives, for all $i\geq i_0$, the existence of a unique $\phi_i \in \mathcal{S}_i$ such that $\Phi_i(\phi_i) = \phi_i$. Therefore
    \begin{equation}\label{eq:forphii}
        (\Dg^k + \a_i^{2k}) \phi_i - (\crit-1)\TBub[i]^{\crit-2} \phi_i =  R_i(\phi_i) + \sum_{j=0}^n \lambda_i^j (\Dg^k + \a_i^{2k})\Zed[i]{j},
    \end{equation}
    and $\phi_i = \Phi_i(\phi_i) \in \Keri^\perp \cap C^{2k-1}(M)$. Moreover, this gives $R_i(\phi_i) \in C^{2k-1}(M)$, so that it follows from \eqref{eq:forphii} that $\phi_i \in C^{2k}(M)$. Finally, we have $\Snorm{k}{\phi_i} = \Snorm{k}{\Phi_i(\phi_i)} \leq C\Snorm{-k}{R_i(\phi_i)} \to 0 $ as $i \to \infty$.

\smallskip\noindent This ends the proof of existence part of Theorem \ref{prop:unicC0}. The bound $\Vert\phi\Vert_*\leq 2C_0C_1$ yields \eqref{bnd:phi} for $l=0$. The control on the derivatives is consequence of Green's representation formula, the control \eqref{bnd:der:G} on the Green's function and the bound \eqref{bnd:phi} for $l=0$.

    \smallskip\noindent This ends the proof  of Theorem \ref{prop:unicC0}.%\end{proof}

\section{Green's function for $\Delta_g^k+\alpha^{2k}$ on $(M,g)$}\label{sec:green:M}
\begin{theorem}\label{th:green:M} Let $k,n\in\nn$ be such that $n>2k\geq 2$ and let $(M,g)$ be a compact Riemannian manifold without boundary. For all $\alpha\geq 1$ and $x\in M$, there exists a unique function $G_\alpha(x,\cdot)\in L^1(M)$ such that
$$\varphi(x)=\int_{M} G_\alpha(x,y)\left(\Delta_g^k\varphi+\alpha^{2k}\varphi\right)(y)\, dv_g(y)$$
for all $\varphi\in C^\infty(M)$. Moreover, the following properties hold:
\begin{itemize} 
\item $G_\alpha(x,\cdot)\in C^\infty(M-\{x\})$ and $G_\alpha(x,y)=G_\alpha(y,x)$ for all $x,y\in M$, $x\neq y$;
\item For any $\alpha\geq 1$, we have that $\lim_{d_g(x,y)\to 0}d_g(x,y)^{n-2k}G_\alpha(x,y)=C_{n,k}$;
\item For all $0\leq l\leq 2k-1$ and for all $q\in\nn$, there exists $C_{q, (M,g),k}>0$ independent of $\alpha\geq 1$ such that
\begin{equation}\label{bnd:der:G}
|\nabla^l_{g,y}G_\alpha(x,y)|\leq C_{q,(M,g),k}\frac{d_g(x,y)^{2k-n-l}}{1+\alpha^q d_g(x,y)^q}\hbox{ for all }x,y\in M\hbox{ s.t. }x\neq y.
\end{equation}
\end{itemize}
\end{theorem}
For the operator $(\Delta_g+\alpha)^k$, similar estimates are proved in Carletti \cite{carletti:green}.

\subsection{Step 1: Green's function for $\Delta_\xi^k+1$ on $\rn$}\label{subsec:green:rn}
We first prove:
\begin{theorem}\label{th:green:rn} Let $k,n\in\nn$ be such that $n>2k\geq 2$. Then there exists a unique function $\Gamma\in L^1(\rn)$ such that
$$\varphi(0)=\int_{\rn} \Gamma(x)\left(\Delta_\xi^k\varphi+\varphi\right)(x)\, dx$$
for all $\varphi\in C^\infty_c(\rn)$. Moreover, the following properties hold:
\begin{itemize} 
\item $\Gamma\in C^\infty(\rn-\{0\})$ and $\Gamma$ is radially symmetrical;
\item The asymptotic as $x\to 0$ is the following:
$$\lim_{x\to 0}|x|^{n-2k}\Gamma(x)=C_{n,k}=\frac{Ga\left(\frac{n}{2}-k\right)}{2^{2k}(k-1)!\pi^{\frac{n}{2}}};\hbox{ where }Ga\hbox{ is the Gamma-function}$$ 
\item For all $l\in\nn$ and for all $q\in\nn$, there exists $C_{q,l,n,k}>0$ such that
\begin{equation}\label{der:Gamma}
|\nabla^l_\xi\Gamma(x)|\leq C_{q,l,n,k}\frac{|x|^{2k-n-l}}{1+|x|^q}\hbox{ for all }x\in \rn-\{0\}.
\end{equation}
\end{itemize}
\end{theorem}
For the operator $(\Delta_\xi+1)^k$, similar properties are proved by Carletti \cite{carletti:green}.

\smallskip\noindent{\it Proof of Theorem \ref{th:green:rn}:} For any function $f\in L^1(\rn)$, we define the Fourier transform 
$$\hat{f}(\xi):=\frac{1}{(2\pi)^{\frac{n}{2}}}\int_{\rn} f(x)e^{-i(x,\xi)}\, dx\hbox{ for all }\xi\in\rn.$$
%We let 
%${\mathcal S}(\rn)$ be the usual Schwartz space of functions whose derivatives are all rapidly decreasing, and 
We investigate $\Gamma\in {\mathcal S}^\prime(\rn)$,  the space of tempered distributions,  such that $(\Delta_\xi^k+1)\Gamma=\delta_0$, the Dirac mass at $0$. Via the Fourier transform, we get 
$$(1+|\xi|^{2k})\hat{\Gamma}=\widehat{(\Delta_\xi^k+1)\Gamma}=\hat{\delta_0}=\frac{1}{(2\pi)^{\frac{n}{2}}},\hbox{ and }\hat{\Gamma}=\frac{1}{(2\pi)^{\frac{n}{2}}}\cdot \frac{1}{1+|\xi|^{2k}}\in L^1_{loc}(\rn).$$
We then define
$$\Gamma:={\mathcal F}^{-1}\left(\frac{1}{(2\pi)^{\frac{n}{2}}}\cdot \frac{1}{1+|\xi|^{2k}}\right)\in{\mathcal S}^\prime(\rn)$$
where ${\mathcal F}^{-1}$ is the inverse of the Fourier transform. For any multi-index $\alpha=(\alpha_1,...,\alpha_n)\in \nn^n$, and $X\in\rn$, we define $X_\alpha=X_{\alpha_1}... X_{\alpha_n}$ and $D_\alpha=\partial_{\alpha_1}...\partial_{\alpha_n}$. For any multi-indices $\alpha,\beta$, we get that
%have that
%$$\widehat{X_\beta D_\alpha \Gamma}=(-i)^{|\alpha|+|\beta|}D_\beta(X_\alpha \hat{\Gamma})=\frac{(-i)^{|\alpha|+|\beta|}}{(2\pi)^{\frac{n}{2}}}D_\beta\left(\frac{X_\alpha}{1+|X|^{2k}}\right).$$
for $|\beta|>n-2k+|\alpha|$, $\widehat{X_\beta D_\alpha \Gamma}\in L^1(\rn)$ and $X_\beta D_\alpha \Gamma\in C^0(\rn)$ and is bounded in $\rn$. Therefore, $\Gamma\in C^\infty(\rn-\{0\})$ and all its derivatives decrease polynomially at infinity.

\smallskip\noindent Let $\eta\in C^\infty_c(B_1(0))$ be such that $\eta(x)=1$ for all $x\in B_{1/2}(0)$. As one checks, since $\Gamma\in C^\infty(\rn-\{0\})$, there exists $f\in C^\infty_c(B_1(0))$ such that $(\Delta_\xi^k+1)(\eta \Gamma)=\delta_0+f\hbox{ in }{\mathcal S}^\prime(\rn)$. Let $G$ be the Green's function for $\Delta_\xi^k+1$ on $B_1(0)$ with Dirichlet boundary condition as in Theorem 6.1 of Robert \cite{robert:pucci-serrin}. Since $G\in C^\infty(B_1(0)-\{0\})$, there exists $h\in C^\infty_c(\rn)$ such that $(\Delta_\xi^k+1)(\eta G)=\delta_0+h$ in ${\mathcal S}^\prime(\rn)$. Therefore $(\Delta_\xi^k+1)(\eta (\Gamma-G))=f-h\in C^\infty_c(\rn)$ in ${\mathcal S}^\prime(\rn)$. Hypoellipticity then yields $\eta (\Gamma-G)\in C^\infty(\rn)$. Since for all $i\in\nn$, there exists $C_{i,n,k}>0$ such that $|\nabla_\xi^i G(x)|\leq C_{i,n,k}|x|^{2k-n-i}$ for all $x\in B_{1/2}(0)$ (see again Robert \cite{robert:pucci-serrin}), we get that there exists $C'_{i,n,k}>0$ such that $|\nabla^i_\xi \Gamma(x)|\leq C'_{i,n,k}|x|^{2k-n-i}$ for all $x\in B_{1/2}(0)$.

\smallskip\noindent Uniqueness arises from $L^1(\rn)\subset {\mathcal S}^\prime(\rn)$ and the expression in term of the Fourier transform. Since the $\Delta_\xi^k(\varphi\circ \sigma)=(\Delta_\xi^k\varphi)\circ\sigma$ for all $\varphi\in C^\infty(\rn)$ and all isometry $\sigma\in O(\rn)$, we get that $\Gamma\circ\sigma^{-1}$ satisfies also Theorem \ref{th:green:rn}, and then $\Gamma=\Gamma\circ\sigma^{-1}$ for all isometry $\sigma$. Therefore $\Gamma$ is radially symmetrical.

\smallskip\noindent All these results prove Theorem \ref{th:green:rn}.\qed

\smallskip\noindent For any $\alpha>0$, defining $\Gamma_\alpha(x):=\alpha^{n-2k}\Gamma(\alpha x)$ for all $x\in\rn-\{0\}$. We get: 

\begin{coro}\label{coro:green:rn:alpha} Let $k,n\in\nn$ be such that $n>2k\geq 2$ and fix $\alpha>0$. Then there exists a unique function $\Gamma_\alpha\in L^1(\rn)$ such that
$$\varphi(0)=\int_{\rn} \Gamma_\alpha(x)\left(\Delta_\xi^k\varphi+\alpha^{2k}\varphi\right)(x)\, dx$$
for all $\varphi\in C^\infty_c(\rn)$. Moreover, the following properties hold:
\begin{itemize} 
\item $\Gamma_\alpha\in C^\infty(\rn-\{0\})$ and $\Gamma_\alpha$ is radially symmetrical;
\item For any $\alpha>0$, we have that $\lim_{x\to 0}|x|^{n-2k}\Gamma_\alpha(x)=C_{n,k}$;
\item For all $l\in\nn$ and for all $q\in\nn$, there exists $C_{q,l,n,k}>0$ independent of $\alpha>0$ such that
\begin{equation*}
|\nabla^l_\xi\Gamma_\alpha(x)|\leq C_{q,l,n,k}\frac{|x|^{2k-n-l}}{1+|\alpha x|^q}\hbox{ for all }x\in \rn-\{0\}.
\end{equation*}
\end{itemize}
\end{coro}

\subsection{Step 2: Construction of an approximate Green's function} We fix $\eta\in C^\infty(\rr)$ such that $\eta(t)=1$ for $t<\delta:=\frac{1}{3}i_g$ and $\eta(t)=0$ for $t>2\delta=\frac{2}{3}i_g$. For $\alpha>0$, since $\Gamma_\alpha$ is radially symmetrical, we define (with a slight abuse of notation):
\begin{equation}\label{def:H:alpha}
H_\alpha(x,y):=\eta(d_g(x,y)) \Gamma_\alpha(d_g(x,y)))\hbox{ for all }x,y\in M,\, x\neq y.
\end{equation}
Note that for all $x\in M$, $y\mapsto H_\alpha(x,y)$ is radial centered at $x$. We get
%We will often write $\nabla$ for the covariant derivative with respect to   $g$. We have that
$$\Delta_{g,y}^k  H_\alpha(x,\cdot)=\eta(d_g(x,y))\Delta_{g,y}^k \Gamma_\alpha(d_g(x,\cdot))+\sum_{i=1}^{2k}\nabla_y^i\eta(d_g(x,\cdot))\star \nabla_y^{2k-i}\Gamma_\alpha(d_g(x,\cdot))$$
where $T\star S$ denotes any linear combination and contraction of two tensors $T$ and $S$. For convenience,   the fixed variable $x$ will often be omitted and the derivatives will all be with respect to $y$. With \eqref{der:Gamma}, we get
$$\Delta_{g}^k  H_\alpha(x,\cdot)=\eta(x,\cdot) \Delta_{g}^k \Gamma_\alpha(x,\cdot)+O\left(\frac{1}{\alpha^q}{\bf 1}_{\delta\leq d_g(x,y)\leq 2\delta}\right).$$
For any radial function $f\in C^2(M)$, in Riemannian polar coordinates, 
\begin{eqnarray*}
\Delta_gf&=&-\frac{1}{r^{n-1}\sqrt{|g|}}\partial_r(r^{n-1}\sqrt{|g|}\partial_r f)=\Delta_\xi f-\frac{\partial_r\sqrt{|g|}}{\sqrt{|g|}}\partial f\\
&=& \Delta_\xi f +  (x)\star\nabla f.
\end{eqnarray*}
%where for any $m\in\zz$, $(x)^m=P(x)\cdot Q(x)$ where $P$ is homogeneous of degree $m$ and $Q$ is smooth. 
Iterating this identity yields
$\Delta_g^kf=\Delta_\xi^k f +\sum_{1\leq i\leq 2k-1} (x)^{i-2(k-1)}\star \nabla^i f$.
Therefore, using that $\Delta_\xi^k\Gamma_\alpha+\alpha^{2k}\Gamma_\alpha=0$ outside $0$ and the pointwise control of the derivatives in Corollary \ref{coro:green:rn:alpha}, we get that
\begin{eqnarray*}
\Delta_{g}^k  H_\alpha(x,\cdot)&=&\eta(x,\cdot) \left( \Delta_\xi^k \Gamma_\alpha +\sum_{1\leq i\leq 2k-1} (x)^{i-2(k-1)}\star \nabla^i \Gamma_\alpha\right)+O\left(\frac{1}{\alpha^q}{\bf 1}_{\delta\leq d_g(x,y)\leq 2\delta}\right)\\
%&=& -\alpha^{2k}\eta(x,\cdot) \Gamma_\alpha (d_g(x,y))+f_\alpha(x,\cdot)\hbox{ in }M-\{x\}\\
&=& -\alpha^{2k}H_\alpha (x,\cdot)+f_\alpha(x,\cdot)\hbox{ in }M-\{x\}
\end{eqnarray*}
where
\begin{equation}\label{bnd:f:alpha}
|f_\alpha(x,y)|\leq C{\bf 1}_{d_g(x,y)\leq 2\delta}\frac{d_g(x,y)^{2-n}}{1+\alpha^q d_g(x,y)^q}\hbox{ for all }x,y\in M,\, x\neq y.
\end{equation}
Therefore, since $\Delta_\xi^k\Gamma_\alpha+\alpha^{2k}\Gamma_\alpha=\delta_0$ in the distribution sense, we get that
\begin{equation}\label{eq:H:alpha}
\left(\Delta_{g}^k +\alpha^{2k}\right) H_\alpha(x,\cdot)=\delta_x+f_\alpha(x,\cdot)\hbox{ distributionally in }M.
\end{equation}

\subsection{Step 3: A Giraud-type Lemma.} 
In the spirit of Giraud's Lemma (see \cite{DHR} and \cite{carletti:green} for the case of a Riemannian manifold), we have that
\begin{lemma}\label{lem:giraud} Let $(M,g)$ be a compact Riemannian manifold of dimension $n\geq 2$ and let $(X_\alpha)_{\alpha>0},(Y_\alpha)_{\alpha>0}: M\times M-\{(x,x)/\, x\in M\}\to\rr$ measurable such that there exists $a,b\in (0,n]$, $p,q> n$ and $C_X,C_Y>0$ independent of $\alpha>0$ such that
\begin{equation*}
|X_\alpha(x,y)|\leq C_X\frac{d_g(x,y)^{a-n}}{1+\alpha^pd_g(x,y)^p}\hbox{ and }|Y_\alpha(x,y)|\leq C_Y\frac{d_g(x,y)^{b-n}}{1+\alpha^qd_g(x,y)^q}
\end{equation*}
for all $x,y\in M$, $x\neq y$ and $\alpha>0$. We set
$$Z_\alpha(x,y):=\int_M X_\alpha(x,z)Y_\alpha(z,y)\, dv_g(z)\hbox{ for all }x,y\in M,\, x\neq y.$$
This definition makes sense and there exists $C(M,g,a,b,p,q)>0$ such that

\smallskip\noindent$\bullet$ If $a+b<n$, then  
$$|Z_\alpha(x,y)|\leq C(M,g,a,b,p,q) C_XC_Y \frac{d_g(x,y)^{a+b-n}}{1+(\alpha d_g(x,y))^{\min\{a+q,b+p\}}}$$

\smallskip\noindent$\bullet$ If $a+b=n$, then 
$$|Z_\alpha(x,y)|\leq C(M,g,a,b,p,q) C_XC_Y \left\{\begin{array}{cc}
|\ln (\alpha d_g(x,y))|&\hbox{ if }\alpha d_g(x,y)<\frac{1}{2}\\
\\
\frac{1}{ (\alpha d_g(x,y))^{\min\{a+q,b+p\}}}&\hbox{ if }\alpha d_g(x,y)\geq \frac{1}{2}\end{array}\right.$$

\smallskip\noindent$\bullet$ If $a+b>n$, then  
$$|Z_\alpha(x,y)|\leq C(M,g,a,b,p,q)  \frac{C_XC_Y}{\alpha^{a+b-n}} \cdot\frac{1}{1+(\alpha d_g(x,y))^{\min\{a+q,b+p\}-a-b+n}}$$
Moreover, if we have that $|Y_\alpha(x,y)|\leq C_Y\frac{(\mu+d_g(x,y))^{b-n}}{1+\alpha^qd_g(x,y)^q}$ for all $x\neq y$ and $\alpha>0$, for some $b\in\rr$ (possibly negative), then, with $a+b<n$ we get
 \begin{equation*}
        |Z_\alpha(x,y)| \leq C \frac{1}{1+(\alpha d_g(x,y))^{\min\{a+q,b+p\}}} \left\{\begin{array}{cc}
            \mu^{n-\gamma} (\mu + d_g(x,y))^{a-n} & \hbox{ if } b<0\\
            (\mu + d_g(x,y))^{a+b-n} & \hbox{ if } b>0
        \end{array}\right. 
    \end{equation*}
for all $x,y\in M$, $x\neq y$ and all $\alpha\geq 1$ and $0<\mu\leq 1$.

\end{lemma}
\noindent{\it Proof of Lemma \ref{lem:giraud}:} We take inspiration of the proof in \cite{DHR}. For $x\neq y$ in $M$, we have that

$$|Z_\alpha(x,y)|\leq C_X C_Y \int_M \frac{d_g(x,z)^{a-n}}{1+\alpha^pd_g(x,z)^p}\frac{d_g(z,y)^{b-n}}{1+\alpha^qd_g(z,y)^q}\, dv_g(z),$$
so the definition makes sense since $a,b>0$. We fix $\delta<\frac{i_g}{4}$.

\smallskip\noindent{\it Case 1: } Let us assume that $d_g(x,y)\geq 2\delta$. Then, for all $z\in B_\delta(x)$, we have that $d_g(z,y)\geq d_g(x,y)-d_g(x,z)\geq \delta$. Arguing similarly on $B_\delta(y)$ and splitting $M$ in three subdomains, we get $C=C(M,g,\delta,a,b,p,q)$ such that
\begin{eqnarray*}
|Z_\alpha(x,y)|&\leq&  C_X C_Y \left(\int_{B_\delta(x)}+\int_{B_\delta(y)}+\int_{\{d_g(z,y)\geq\delta\hbox{ and }d_g(z,x)\geq\delta\}}  \right)\\
%&\leq & C_X C_Y \left(\int_{B_\delta(x)} \frac{d_g(x,z)^{a-n}}{1+\alpha^pd_g(x,z)^p}\frac{dv_g(z)}{\alpha^q}+\int_{B_\delta(y)} \frac{d_g(z,y)^{b-n}}{1+\alpha^qd_g(z,y)^q}\frac{dv_g(z)}{\alpha^p}+\frac{C}{\alpha^{p+q}} \right)\\
%&\leq & C_X C_Y C \left(\frac{1}{\alpha^q} \int_{0}^\delta \frac{r^{a-1}}{1+\alpha^pr^p}\, dr+\frac{1}{\alpha^p} \int_{0}^\delta \frac{r^{b-1}}{1+\alpha^qr^q}\, dr+\frac{1}{\alpha^{p+q}} \right)\\
%&\leq & C_X C_Y C\left(\frac{1}{\alpha^{a+q}} \int_{0}^{\delta\alpha} \frac{s^{a-1}}{1+s^p}\, ds+\frac{1}{\alpha^{b+p}} \int_{0}^{\delta\alpha} \frac{s^{b-1}}{1+s^q}\, ds+\frac{1}{\alpha^{p+q}} \right)\\
&\leq & C_X C_Y C\left(\frac{1}{\alpha^{a+q}}  +\frac{1}{\alpha^{b+p}}  +\frac{1}{\alpha^{p+q}} \right)\leq \frac{C_X C_Y C}{\alpha^{\min\{a+q,b+p\}}}
\end{eqnarray*}
since $p,q>n$,  $0<a,b\leq n$ and $\alpha\geq 1$. This proves the Lemma when $d_g(x,y)\geq 2\delta$.

\smallskip\noindent {\it Case 2: } We assume that $d_g(x,y)<2\delta$. Let $Y_0\in \rn$ be such that $y=\hbox{exp}_x(Y_0)$. In particular $0<|Y_0|< 2\delta$. For $z\in M-B_{3\delta}(x)$, we have that $d_g(y,z)\geq d_g(z,x)-d_g(x,y)\geq \delta$. The change of variable $z:=\hbox{exp}_x^g(Z)$ yields
\begin{eqnarray*}
|Z_\alpha(x,y)|&\leq&  C_X C_Y \left(\int_{B_{3\delta}(x)}+\int_{M-B_{3\delta}(x)} \right)\\
%&\leq &  C_X C_Y C \left(\int_{B_{3\delta}(0)} \frac{|Z|^{a-n}}{1+\alpha^p|Z|^p}\cdot\frac{d_g(\hbox{exp}_x(Z),\hbox{exp}_x(Y_0))^{b-n}}{1+\alpha^qd_g(\hbox{exp}_x(Z),\hbox{exp}_x(Y_0))^q}\, dZ+\frac{1}{\alpha^{p+q}} \right)\\
&\leq &  C_X C_Y C \left(\int_{B_{3\delta}(0)} \frac{|Z|^{a-n}}{1+\alpha^p|Z|^p}\cdot\frac{|Z-Y_0|^{b-n}}{1+\alpha^q|Z-Y_0|^q}\, dZ+\frac{1}{\alpha^{p+q}} \right)\\
 \end{eqnarray*}
We set $\Theta_0:=\frac{Y_0}{|Y_0|}$ so that $|\Theta_0|=1$. The change of variable $Z=|Y_0|X$ yields
\begin{multline*}
|Z_\alpha(x,y)|\\\leq   C_X C_Y C \left(\int_{B_{\frac{3\delta}{|Y_0|}}(0)} \frac{|X|^{a-n}|Y_0|^{a+b-n}}{1+\alpha^p|Y_0|^p|X|^p}\cdot\frac{|X-\Theta_0|^{b-n}}{1+\alpha^q|Y_0|^q|X-\Theta_0|^q}\, dX+\frac{1}{\alpha^{p+q}} \right)\label{ineq:Za:1}
 \end{multline*}
In order to get the result, one needs to separate the cases $\alpha|Y_0|\leq  \frac{1}{3}$, $\frac{1}{3}\leq \alpha|Y_0|\leq 3$, and $\alpha|Y_0|\geq 3$, with the distinction $a+b>n$, $a+b=n$, and $a+b<n$.

\subsection{Step 4: Construction of the Neumann series} We fix $q\geq 1$ and an integer $N\geq E\left(\frac{n}{2}\right)+\frac{q}{2}$. For $x,y\in M$, $x\neq y$ and $\alpha\geq 1$, as in   \cite{robert:green}, we set
\begin{equation}\label{def:G:alpha}
G_\alpha(x,y):=H_\alpha(x,y)+\sum_{i=1}^N\int_M H_\alpha^i(x,z)H_\alpha(z,y)\, dv_g(z)+\phi_\alpha(x,y)
\end{equation}
where $H_\alpha$ is as in \eqref{def:H:alpha}, $\phi_\alpha(x,\cdot)\in C^{2k}(M)$ for all $x\in M$ and the $H_\alpha^i$'s  are defined recursively as follows:
$$\left\{\begin{array}{cc}
H_\alpha^1(x,y):= -f_\alpha(x,y)&\\
H_\alpha^{i+1}(x,y)=-\int_{M}H_\alpha^i(x,z)f_\alpha(z,y)\, dv_g(z)&\hbox{for all }i\geq 0\end{array}\right.$$
for all $x,y\in M$, $x\neq y$, where $f_\alpha$ is as in \eqref{eq:H:alpha} and \eqref{bnd:f:alpha}.  Iterating  Lemma \ref{lem:giraud}  yields
%depending on the parity of $n$, we get that
\begin{equation*}
\hbox{for }i<E\left(\frac{n}{2}\right)\hbox{ or }\{i=E\left(\frac{n}{2}\right),\, n\hbox{ odd}\},\;|H_\alpha^i(x,y)|\leq C\frac{d_g(x,y)^{2i-n}}{1+(\alpha d_g(x,y))^q}
\end{equation*}
\begin{equation*}
\hbox{for }\{i=E\left(\frac{n}{2}\right),\, n\hbox{ even}\},\,|H_\alpha^{i}(x,y)|\leq C\left\{\begin{array}{cc}
|\ln (\alpha d_g(x,y))| &\hbox{ if }\alpha d_g(x,y)<\frac{1}{2}\\
\frac{1}{(\alpha d_g(x,y))^q}&\hbox{ if }\alpha d_g(x,y)\geq \frac{1}{2}
\end{array}\right.
\end{equation*}
\begin{equation}\label{est:H:N}
\hbox{for }i>E\left(\frac{n}{2}\right) ,\;|H_\alpha^i(x,y)|\leq \frac{C}{\alpha^{2(i-E(n/2))-1}}\cdot \frac{1}{1+(\alpha d_g(x,y))^q}
\end{equation}
for all $x,y\in M$, $x\neq y$ and $\alpha\geq 1$.
Noting that for any $\alpha\geq 1$ and $x\in M$, we have that $H_\alpha^{N+1}(x,\cdot)\in C^{0,\theta}(M)$ for some $0<\theta<1$ (see for instance \cite{DHR}), we let $\phi_\alpha(x,\cdot)\in C^{2k}(M)$ be such that
\begin{equation}\label{def:phi:a}
\Delta_g^k \phi_\alpha(x,\cdot)+\alpha^{2k}\phi_\alpha(x,\cdot)=H_\alpha^{N+1}(x,\cdot)\hbox{ in }M.
\end{equation}
The standard cancellation in the Neumann series (see for instance \cite{robert:gjms}) yields
%\begin{equation*}
%\Delta_g^k G_\alpha(x,\cdot)+\alpha^{2k}G_\alpha(x,\cdot)=\delta_x+(\Delta_g^k +\alpha^{2k})\phi_\alpha(x,\cdot)-H_\alpha^{N+1}(x,\cdot)\hbox{ weakly in }M.
%\end{equation*}
%
%Therefore, we have that
\begin{equation}\label{eq:G:dist}
\Delta_g^k G_\alpha(x,\cdot)+\alpha^{2k}G_\alpha(x,\cdot)=\delta_x\hbox{ weakly in }M.
\end{equation}
Since $\Delta_g^k+\alpha^{2k}$ is self-adjoint, we get that $G_\alpha(x,y)=G_\alpha(y,x)$ for all $x,y\in M$ such that $x\neq y$.
\subsection{Step 5: Regularity for $\Delta_g^k+\alpha^{2k}$}
\begin{lemma}\label{lem:regul} Let  $(h_\alpha)_{\alpha\geq 1}\in L^\infty(M)$ be a family of functions. For all $\alpha\geq 1$, let $\varphi_\alpha\in H_k^2(M)$ be the variational solution to 
\begin{equation}\label{eq:phi:alpha}
\Delta_g^k \varphi_\alpha+\alpha^{2k}\varphi_\alpha=h_\alpha\hbox{ in }M.
\end{equation}
Assume that there exists $p\in M$, $C_H>0$ and $q>\frac{n+2k}{2}$ such that $|h_\alpha(x)|\leq C_H(1+(\alpha d_g(x,p))^{-q}$  for all $x\in M$ and $\alpha\geq 1$. Then there exists $C=C(M,g,k)>0$ such that for all $0\leq i<2k$, we have that $|\nabla^i_g\varphi_\alpha(x)|_g\leq  C\cdot C_H \alpha^{i-2k}$ for all $x\in M$ and $\alpha\geq 1$.\end{lemma}
\noindent{\it Proof of Lemma \ref{lem:regul}:} Since $q>\frac{n+2k}{2}$, a straightforward estimate yields $\Vert h_\alpha\Vert_{\frac{2n}{n+2k}}\leq C\cdot C_H \alpha^{-\frac{n+2k}{2}}$ for all $\alpha\geq 1$. Since $\alpha\geq 1$, the operator $\Delta_g^k+\alpha^{2k}$ is uniformly coercive and then, using Sobolev's inequality \eqref{ineq:AB}, we have that
\begin{eqnarray*}
\Vert \varphi_\alpha\Vert_{H_k^2}^2&\leq& C\int_M \left((\Delta_g^\frac{k}{2}\varphi_\alpha)^2+\varphi_\alpha^2\right)\, dv_g\leq \int_M \left((\Delta_g^\frac{k}{2}\varphi_\alpha)^2+\alpha^{2k}\varphi_\alpha^2\right)\, dv_g\\
&\leq & \int_M (\Delta_g^k\varphi_\alpha+\alpha^{2k}\varphi_\alpha)\varphi_\alpha\, dv_g=\int_M h_\alpha\varphi_\alpha\, dv_g\leq \Vert h_\alpha\Vert_{\frac{2n}{n+2k}}\Vert \varphi_\alpha\Vert_{\crit}\\
&\leq& C \Vert h_\alpha\Vert_{\frac{2n}{n+2k}}\Vert \varphi_\alpha\Vert_{H_k^2}
\end{eqnarray*} 
and therefore $\Vert \varphi_\alpha\Vert_{\crit}\leq C\Vert \varphi_\alpha\Vert_{H_k^2}\leq C \Vert h_\alpha\Vert_{\frac{2n}{n+2k}}\leq C\cdot C_H \alpha^{-\frac{n+2k}{2}}$ for all $\alpha\geq 1$. We fix a family $(x_\alpha)_{\alpha\geq 1}\in M$ and we define $\tilde{\varphi}_\alpha(X):=\varphi_\alpha(\hbox{exp}_{x_\alpha}(\alpha^{-1}X))$ for all $X\in B_{\alpha i_g}(0)\subset\rn$. Equation \eqref{eq:phi:alpha} rewrites $\Delta_{g_\alpha}^k\tilde{\varphi}_\alpha+\tilde{\varphi}_\alpha=\tilde{h}_\alpha(X):=\alpha^{-2k}h_\alpha(\hbox{exp}_{x_\alpha}(\alpha^{-1}X))$,
where $g_\alpha:=\hbox{exp}_{x_\alpha}^\star g(\alpha^{-1}\cdot)$. We fix $R>0$. We get
\begin{eqnarray*}
\Vert \tilde{\varphi}_\alpha\Vert_{L^{\crit}(B_R(0))}&=&\left(\int_{B_R(0)}|\tilde{\varphi}_\alpha|^{\crit}\, dX\right)^{\frac{1}{\crit}}\leq C\alpha^{\frac{n-2k}{2}}\left(\int_{B_{R\alpha^{-1}}(x_\alpha)}| \varphi _\alpha|^{\crit}\, dv_g\right)^{\frac{1}{\crit}}\\
&\leq& C\alpha^{\frac{n-2k}{2}}\Vert \varphi_\alpha\Vert_{\crit}\leq C\cdot C_H \alpha^{-2k}.
\end{eqnarray*}
We fix $p>1$ large enough. We fix $i\leq 2k-1$. Sobolev's embedding and elliptic regularity (see  \cite{ADN} or the Appendix in \cite{robert:gjms}) yield
\begin{eqnarray*}
|\nabla^i_\xi\tilde{\varphi}_\alpha(0)|&\leq & C\Vert \tilde{\varphi}_\alpha\Vert_{C^{2k-1}(B(0,R/2))}\leq  C\Vert \tilde{\varphi}_\alpha\Vert_{H_{2k}^p(B(0,R/2))}\\
%&\leq & C\left(\Vert \tilde{h}_\alpha\Vert_{L^p(B(0,R))}+\Vert \tilde{\varphi}\Vert_{L^{\crit}(B(0,R))}\right)\\
&\leq & C\left(\Vert \tilde{h}_\alpha\Vert_{L^\infty(B(0,R))}+\Vert \tilde{\varphi}_\alpha\Vert_{L^{\crit}(B(0,R))}\right)\\
&\leq& C\cdot C_H\left(\alpha^{-\frac{n+2k}{2}}+\alpha^{-2k}\right)\leq C\cdot C_H\ \alpha^{-2k}.
\end{eqnarray*}
Coming back to $\varphi_\alpha$, we get that $|\alpha^{-i}\nabla_g^i\varphi_\alpha(x_\alpha)|\leq C\sum_{j\leq i}|\nabla^j_\xi\tilde{\varphi}_\alpha(0)|C\cdot C_H\ \alpha^{-2k}$. This proves the Lemma.\qed

\subsection{Step 6: Pointwise estimates and conclusions.}
We fix $q>\frac{n+2k}{2}$ and we define $N\geq E\left(\frac{n}{2}\right) +\frac{q}{2}$ as in the construction. We let $\phi_\alpha$ be as in \eqref{def:phi:a}. With the pointwise control \eqref{est:H:N} and Lemma \ref{lem:regul}, we get that
\begin{equation*}%\label{ineq:psi}
|\nabla_y^i\phi_\alpha(x,\cdot)|\leq   \frac{C\alpha^i }{\alpha^{2(N -E(n/2))+1+2k}} \leq \frac{C\alpha^i }{\alpha^{q+1 +2k}}\leq \frac{C }{\alpha^{q+1}} \hbox{ for all }i<2k.
\end{equation*}
With \eqref{def:H:alpha} and the controls \eqref{est:H:N}, we get that for all $1\leq i\leq N$ 
\begin{equation*}
\left|\int_M H_\alpha^i(x,z)H_\alpha(z,y)\, dv_g(z)\right|\leq C \frac{d_g(x,y)^{2k-n}}{1+(\alpha d_g(x,y))^q}
\end{equation*}
for all $x,y\in M$, $x\neq y$. With the definition \eqref{def:G:alpha}, we get that
$$|G_\alpha(x,y)|\leq C \frac{d_g(x,y)^{2k-n}}{1+(\alpha d_g(x,y))^q}+\frac{C}{\alpha^{q}}\leq C \frac{d_g(x,y)^{2k-n}}{1+(\alpha d_g(x,y))^q}.$$
The upper bound of the derivatives is similar, and we get Theorem \ref{th:green:M}.
%
%Concerning derivatives, as one checks, for all $l\leq 2k-1$, we have that
%$$|\nabla_y^l H_\alpha(x,y)|\leq C \frac{d_g(x,y)^{2k-n-l}}{1+(\alpha d_g(x,y))^q}\hbox{ for all }x\neq y\in M\hbox{ and }\alpha\geq 1.$$
%Therefore, we get that for $1\leq i\leq N$,
%$$|\nabla_y^l (H_\alpha^i(x,z)H_\alpha(z,y))|\leq C \frac{d_g(x,z)^{2-n}}{1+(\alpha d_g(x,z))^q}\cdot \frac{d_g(z,y)^{2k-n-l}}{1+(\alpha d_g(z,y))^q}$$
%for all $x,y,z\in M$, $z\neq x,y$  and $\alpha\geq 1$. It then follows from these inequalities, integration theory and the definition    \eqref{def:G:alpha} of $G_\alpha$ that $G_\alpha(x,\cdot)\in C^{2k-1}(M-\{x\})$ and that
%$$|\nabla_y^l G_\alpha(x,y)|\leq C \frac{d_g(x,y)^{2k-n-l}}{1+(\alpha d_g(x,y))^q}\hbox{ for all }x\neq y\in M\hbox{ and }\alpha\geq 1.$$
%This inequality and \eqref{eq:G:dist} then yield Theorem \ref{th:green:M}. 

\end{document}